\documentclass[11pt, a4paper, reqno]{amsart}

\usepackage{kantlipsum} 
\usepackage[margin=2.5cm]{geometry}
\allowdisplaybreaks

\usepackage{setspace}

\usepackage[utf8]{inputenc}
\usepackage[T1]{fontenc}
\usepackage{color}
\usepackage{tikz}
\usepackage{float}
\usepackage[normalem]{ulem}
\newcommand{\stkout}[1]{\ifmmode\text{\sout{\ensuremath{#1}}}\else\sout{#1}\fi}

\usepackage[foot]{amsaddr}

\usepackage[
colorlinks,
pdfpagelabels,
pdfstartview = FitH,
bookmarksopen = true,
bookmarksnumbered = true,
linkcolor = blue,
plainpages = false,
hypertexnames = false,
citecolor = red] {hyperref}

\hypersetup{
linktoc=page
}

\usepackage[url=false, style=numeric, backend=biber,firstinits=true, maxbibnames=99, maxalphanames=4]{biblatex}
\DeclareFieldFormat{issn}{}
\usepackage[capitalize]{cleveref}
\crefname{enumi}{}{}
\crefname{equation}{}{}
\crefname{assumption}{Assumption}{Assumptions}

\newtheorem{theorem}{Theorem}

\newcommand{\raisedtarget}[1]{%
  \raisebox{\fontcharht\font`P}[0pt][0pt]{\hypertarget{#1}{}}%
}

\newtheorem{proposition}{Proposition}[section]
\newtheorem{lemma}[proposition]{Lemma}
\newtheorem{corollary}[proposition]{Corollary}
\theoremstyle{definition}
\newtheorem{definition}[proposition]{Definition}
\newtheorem{remark}[proposition]{Remark}
\newtheorem{example}[proposition]{Example}
\newtheorem{assumption}{Assumption}
\numberwithin{equation}{section}

\usepackage{amssymb, mathrsfs}
\def \R {\mathbb{R}}

\def \T {\mathbb{T}}
\def \rmH {\mathrm{H}}
\def \L {\mathscr{L}}
\def \rmL {\mathrm{L}}
\def \rmW {\mathrm{W}}
\def \rmd {\mathrm{d}}
\def \Tr {\mathrm{Tr}}
\def \calD {\mathcal{D}}
\def \calU {\mathcal{U}}
\def \calV {\mathcal{V}}
\def \bfA {\mathbf{A}}

\def \tr {\mathrm{tr}}

\def \rmC {\mathrm{C}}

\newcommand{\ssubset}{\subset\joinrel\subset}

\usepackage{mathtools}
\DeclarePairedDelimiter\abs{\lvert}{\rvert}
\DeclarePairedDelimiter\norm{\lVert}{\rVert}
\DeclarePairedDelimiter{\pair}{\langle}{\rangle}

\makeatletter
\def\@tocline#1#2#3#4#5#6#7{\relax
\ifnum #1>\c@tocdepth 
\else
\par \addpenalty\@secpenalty\addvspace{#2}%
\begingroup \hyphenpenalty\@M
\@ifempty{#4}{%
\@tempdima\csname r@tocindent\number#1\endcsname\relax
}{%
\@tempdima#4\relax
}%
\parindent\z@ \leftskip#3\relax \advance\leftskip\@tempdima\relax
\rightskip\@pnumwidth plus4em \parfillskip-\@pnumwidth
#5\leavevmode\hskip-\@tempdima
\ifcase #1
\or\or \hskip 1em \or \hskip 2em \else \hskip 3em \fi%
#6\nobreak\relax
\dotfill\hbox to\@pnumwidth{\@tocpagenum{#7}}\par
\nobreak
\endgroup
\fi}
\makeatother

\begin{document}


\title[Weak and Perron Solutions]{Weak and Perron Solutions for Stationary Kramers-Fokker-Planck Equations in Bounded Domains}

\author{Benny Avelin}
\address[B.~Avelin]{Department of Mathematics, Uppsala University\\
    751 05 Uppsala, Sweden}
\email{benny.avelin@math.uu.se}

\author{Mingyi Hou}
\address[M.~Hou]{Department of Mathematics, Uppsala University\\
    751 05 Uppsala, Sweden}
\email{mingyi.hou@math.uu.se}

\date{\today}

\begin{abstract}
    In this paper, we investigate weak solutions and Perron-Wiener-Brelot solutions to the linear stationary Kramers-Fokker-Planck equation in bounded domains.
    We establish the existence of weak solutions in product domains by applying the Lions-Lax-Milgram theorem and the vanishing viscosity   method. Furthermore, we show that these solutions coincide in well-behaved domains.
    Building on the existence of weak solutions in product domains, we develop the foundational theory of Perron-Wiener-Brelot solutions in     arbitrary bounded domains.
    Our results rely on recent advancements in the theory of kinetic Fokker-Planck equations with rough coefficients.
\end{abstract}

\subjclass[2020]{35Q84, 35D30, 35D99, 35J25; 35H10, 35J70, 31C45}
\keywords{Kramers-Fokker-Planck, hypoelliptic, weak solution, trace problem, PWB solution, Dirichlet problem.}

\maketitle

\tableofcontents


\section{Introduction}
The general Kramers-Fokker-Planck equation in divergence form is given by
\begin{equation} \label{eq:kfp}
    \mathscr{L} u := \nabla_v\cdot (\mathbf{A} \nabla_v u) + \mathbf{b}\cdot\nabla_v u + v\cdot\nabla_x u = 0,
\end{equation}
where $(x, v)\in \mathbb{R}^n\times\mathbb{R}^n$, with $n\geq 1$ an integer. Here,
$\mathbf{A}=\mathbf{A}(x,v)$ is a real symmetric matrix and $\mathbf{b} = \mathbf{b}(x,v)$ is a vector field.
We will sometimes denote a point in $\mathbb{R}^{n}\times\mathbb{R}^n$ as $\xi=(x, v)$.
This operator appears in special cases, as the generator of a kinetic stochastic process, the corresponding parabolic equation $\partial_t \varrho = \mathscr{L}^\ast \varrho$ is the kinetic Fokker-Planck equation for the density of the process, where $\mathscr{L}^\ast$ is the adjoint operator of $\mathscr{L}$.

\begin{assumption}\label{assump:1}
    Throughout the paper, we assume that $\mathbf{A} = (a_{ij})_{1\leq i,j\leq n}$ is a symmetric matrix, that $\mathbf{b} = (b_1,\dots, b_n)$ is a vector field, and that $a_{ij}, b_i : \R^{2n} \to \mathbb{R}$ are both measurable. Moreover, they satisfy the following conditions:
    \begin{equation*}
        \mathbf{A} \geq \lambda \mathbf{I},\,\abs{a_{ij}} \leq \Lambda \text{ for } i,j=1,\dots,n \text{ and } \abs{\mathbf{b}}\leq \nu,
    \end{equation*}
    for some constants $0<\lambda\leq \Lambda$ and $\nu\geq 0$, where $\mathbf{I}$ denotes the identity matrix.
\end{assumption}

In this paper, we will study weak and Perron-Wiener-Brelot solutions for the stationary Kramers-Fokker-Planck \cref{eq:kfp} in bounded domains.

Kramers-Fokker-Planck equations (or Kolmogorov equations) appear in the theory of stochastic processes.
A prototype of such equations is as the generator of the following process, let $(X_t, V_t)\in\R\times\R$ be stochastic processes satisfying the stochastic differential equation
\begin{equation*}
    \begin{cases}
        \mathrm{d}X_t = V_t\,\mathrm{d}t, \\
        \mathrm{d}V_t = \mathrm{d} B_t,
    \end{cases}
\end{equation*}
where $B_t$ is the standard Brownian motion in $\R$.
Its generator is $\frac{1}{2}\partial_{v}^2 + v\cdot\partial_x $,
and the density $\varrho(t,x,v)$ satisfies the dual equation (forward Fokker-Planck)
$\partial_t \varrho = \frac{1}{2}\partial_v^2 \varrho - v\cdot\partial_x \varrho$,
which was studied by Kolmogorov back in the 30s,~\cite{Kol}.

Another example that motivates our study is the kinetic process,
\begin{equation*}
    \begin{cases}
        \mathrm{d} X_t = V_t\, \mathrm{d} t, \\
        \mathrm{d} V_t = - (V_t +\nabla_x U(x))\,\mathrm{d} t + \sqrt{2}\,\mathrm{d} B_t,
    \end{cases}
\end{equation*}
where $U(x)$ is a smooth function, such as a quadratic potential, and $B_t$ is the standard Brownian motion in $\R^n$.
It is well known that the generator of this process is given by
\begin{equation}\label{eq:kfpvillani}
    (\Delta_v -v\cdot \nabla_v ) + (-\nabla_x U\cdot\nabla_v + v\cdot \nabla_x ).
\end{equation}
This process appears naturally in various fields of kinetic theory and statistical physics, including plasma physics, condensed matter physics, and more recently, in machine learning and the optimization of deep neural networks using the method of stochastic gradient descent with momentum.

One of our main motivations for developing the solvability of the Dirichlet problem, particularly for the stationary problem related to \cref{eq:kfp}, is to develop the necessary toolbox to establish an Eyring-Kramers formula for the exit time from a metastable equilibrium, as in~\cite{BEGK}; see also~\cites{AJV,AJ}.
Another approach involves the so-called quasi-stationary distribution, which requires proving existence in bounded position and unbounded velocity, as done using probabilistic methods in~\cites{LRR1,LRR2}.

The Kramers-Fokker-Planck equation has been extensively studied from various perspectives.
For instance, a comprehensive introduction to non-degenerate Kolmogorov equations from a PDE perspective can be found in~\cite{BKRS15}.
It is well established that, under certain conditions on the coefficient matrices, the Kolmogorov operator is hypoelliptic in the terminology of H\"ormander,~\cite{Hor67}.
In the context of hypocoercivity, \cref{eq:kfpvillani} has been extensively studied starting with the works of Villani~\cite{Vil09}; see, for instance,~\cites{HN04, HHS11, DMS15, AAMN21, BPM22, CLW23}, among others.

In recent years, a lot of attention has been given to the study of Kolmogorov equations with rough coefficients.
This started with the development of the De Giorgi-Nash-Moser theory originally in~\cite{PP04}, see also~\cite{WZ11}. Recently the Harnack inequality was proved in the case of rough coefficients in~\cite{GIMV19}, which sparked a series of works on the subject; see, for instance~\cites{GM22,GI23,AR22,LN22,GN22,Zhu22,Sil22}.

Originally, the existence of weak solutions for the time-dependent case was established in~\cite{Car99}. Later, in~\cite{AAMN21}, a variational approach was developed to prove the existence of weak solutions on the torus, which was later extended to bounded domains in the time-dependent case in~\cite{LN21}.
Noteworthy is~\cite{NZ21}, which establishes weak solutions in the whole space for the fractional version of the degenerate Kolmogorov equation.
In~\cite{AHN23}, a semi-spectral-Galerkin method was developed to prove the existence of weak solutions in the time-dependent case on the torus.
An alternative approach in~\cite{AIN24weak} established the existence of weak solutions in the whole space for equations with rough coefficients, avoiding the use of De Giorgi-Nash-Moser theory and also covering the fractional case.
Finally, the concurrent works~\cites{Zhu22,Sil22} appeared: the former established the existence of weak solutions in the time-dependent case, while the latter addressed boundary point regularity using the renormalized solution approach of DiPerna-Lions~\cite{DL89}.

In the first part of the paper, we combine the method developed in~\cite{Zhu22} with~\cite{Sil22} and extend the existence result to the stationary case.
More importantly, we collect the current knowledge about different notions of trace and conclude that for the degenerate Kolmogorov equations, a weak trace theory holds, which is enough to establish a comparison principle. This improved understanding serves as a solid foundation for future studies on the weak solution theory for degenerate Kolmogorov equations.

In the second part of the paper, we develop the theory of Perron-Wiener-Brelot solutions, including resolutivity (in the non-divergence case), which means that the upper and lower Perron solutions coincide and thus provides a unique Perron solution to the Dirichlet problem.

\subsection{The trace problem}
Concerning the theory of weak solutions to the boundary value problem for the Kramers-Fokker-Planck equations, a well-known open problem is the so-called trace problem; for more details, see also~\cites{Sil22,AAMN21}.
To describe the problem briefly, let $\mathcal{U}\subset\mathbb{R}^n$ and $\mathcal{V}\subset\mathbb{R}^n$ be sets for $x$ and $v$, respectively.
Consider the function space $\rmH^1_{\mathrm{hyp}} (\mathcal{U}\times\mathcal{V})$, defined as $\rmH^1_{\mathrm{hyp}} (\mathcal{U}\times\mathcal{V}):= \{ u(x, v)\in \rmL^2(\mathcal{U}; \rmH^1(\mathcal{V})) \textrm{ s.t. } v\cdot\nabla_x u(x,v) \in \rmL^2(\mathcal{U}; \rmH^{-1}(\mathcal{V})) \}$ where $\rmH^{-1}(\mathcal{V})$ is the dual space of $\rmH^1_0(\mathcal{V})$.
It is known only in the case $n=1$ that there exists a trace (in the classical sense) for the space $\rmH^1_\mathrm{hyp}$, as shown in~\cite{BG68}.
In higher dimensions, due to the transport term $v\cdot\nabla_x $, we lack information near the set $\{(x,v)\in\partial\mathcal{U}\times\mathcal{V}:v\cdot\mathbf{n}_x = 0\}$, where $\mathbf{n}_x$ is the unit outer normal of $\mathcal{U}$.
Several attempts have been made towards a trace theorem; see, for example, the discussion in~\cites{AAMN21}, but the problem remains open.
Nevertheless, one can establish the theory of weak solutions by avoiding the boundary, as demonstrated in~\cites{AAMN21, AHN23}, by considering a space like $\T^n\times\mathbb{R}^n$, where $\T^n$ denotes the $n$-dimensional torus.
A full existence and uniqueness result using the variational approach is later given in~\cite{LN21}, where the notion of a trace can be avoided.

A weaker notion of trace indeed dates back to~\cite{Mis00}, where Mischler deals with the $\rmL^1$ solutions to Vlasov equations in bounded position and unbounded velocity domain by implementing the convolution-translation and the renormalization from~\cite{DL89}.
Later on, in~\cite{Mis10}, the technique is generalized to other kinetic equations, for example, Boltzmann, Vlasov-Possion, and Fokker-Planck type equations.
Recently, Silvestre~\cite{Sil22} discovered that a weaker notion of trace can be defined for the space $\rmH_{\mathrm{hyp}}^1$, which is canonical for an $\rmL^2$ weak solution theory.
We will continue with this and elaborate on the definition of this trace.
We formally define it as the weak trace, \cref{thm:weaktrace}, as opposed to the classical trace, which is stronger.

\section{Novelty and contributions}

This paper is a continuation of the recent developments in the theory of weak solutions to the Kramers-Fokker-Planck equation with rough coefficients. We work with the more natural notion of weak solutions that does not a priori involve any trace function, in contrast to~\cite{Zhu22}. Instead, we use the notion of weak trace discovered by~\cite{Sil22}, and the convolution-translation technique to prove that a renormalized Green's formula still holds, provided the domain is nice enough. Comparing this to~\cite{Zhu22} where they construct solutions which are nice enough for the renormalized Green's formula to hold. The first main contribution from our side is that in-fact the renormalized Green's formula is always true for weak solutions, as the weak trace is indeed enough. Using the renormalized Green's formula we prove a weak comparison principle working with the weak trace. This allows us to for example show that solutions obtained using functional analytic techniques and the ones constructed using the vanishing viscosity method are the same in product domains that are $\rmC^{1,1}$ in $x$.

In the second part of the paper we use the notion of weak solutions to prove the existence of Perron solutions and resolutivity (in the non-divergence case) in arbitrary bounded domains. The method here is inspired by nonlinear parabolic potential theory where there is no access to any representation formula or the existence of solutions in heat balls or similar.
The main reason for taking this approach for the stationary problem is that we have very little information of the Greens function or fundamental solution if the operator doesn't have a nice invariant Galilean transformation. The second reason is that due to the rough coefficients, even in the time-dependent case we cannot use any explicit Greens functions. The consequence is that we cannot study the potential theory using Doob's $\beta$-harmonic spaces directly, as initially done in~\cite{Man97} for H\"older continuous coefficients.
For more information, see, for instance, the overview~\cite{AP20} as well as the recent works~\cites{Kog17,Kog19,Kog20,KLT18}.
Instead, as mentioned above we borrow the methodology from nonlinear parabolic potential theory for the $p$-Laplacian~\cite{KL96}, which relies on the existence of weak solutions in velocity-space (stationary) cylinders (not a Doob basis), and De Giorgi-Nash-Moser (regularity) theory.

\section{Main results}\label{sec:mainresults}
First, we investigate weak solutions in product domains, whereby \emph{domain} we refer to an open and connected set.

\begin{assumption}\label{assump:D0}
    Throughout the paper, we denote
    \begin{equation*}
        \calD = \calU\times \calV
    \end{equation*}
    where $\calU,\calV\subset\R^n$ are bounded domains such that, $\partial \calU$, $\partial \calV$ are $\rmC^{0,1}$ (Lipschitz).
\end{assumption}
The boundary of $\calD$ is split into the following parts:
\begin{equation*}
    \partial^v \calD := \calU\times\partial \calV, \textrm{ and } \partial^x \calD := \partial \calU \times \calV.
\end{equation*}
$\partial^x \calD$ is further split into positive, singular and negative parts:
\begin{equation}\label{def:weak:bdy}
    \begin{split}
        \partial^x_{+}\calD := \left\{ (x,v)\in\partial^x \calD:  v\cdot\mathbf{n}_x > 0 \right\}, \\
        \partial^x_{0}\calD := \left\{ (x,v)\in\partial^x \calD:  v\cdot\mathbf{n}_x = 0 \right\}, \\
        \partial^x_{-}\calD := \left\{ (x,v)\in\partial^x \calD:  v\cdot\mathbf{n}_x < 0 \right\}.
    \end{split}
\end{equation}
Here, $\mathbf{n}_x$ denotes the unit outward normal vector of $\calU$.
Then the \emph{hypoelliptic boundary} of $ \calD$ is defined to be
\begin{equation*}
    \partial_{\mathrm{hyp}}  \calD := \overline{\partial^x_+ \calD\cup\partial^x_0 \calD\cup\partial^v \calD}.
\end{equation*}

Below are some notation conventions.
Recall that a point $(x, v)\in \R^{2n}$ is denoted as $\xi = (x, v)$.
We will sometimes write $\nabla_\xi = (\nabla_x , \nabla_v)$.

Let $\Omega\subset\R^{2n}$ be a domain, we define the standard Sobolev space
\begin{equation*}
    \rmH^1(\Omega) := \{ f\in \rmL^2(\Omega): \nabla_\xi f\in \rmL^2(\Omega)\},\,\norm{f}_{\rmH^1(\Omega)} = \left(\int_\Omega f^2 +\abs{\nabla_\xi f}^2 \,\mathrm{d}\xi\right)^{1/2},
\end{equation*}
where $\nabla_\xi$ is taken in the weak sense.
The space of $\rmH^1$ functions that vanishes on the boundary is denoted by $\rmH^1_0 (\Omega):=\{f\in \rmH^1(\Omega): \Tr(f) = 0\}$
where $\Tr: \rmH^1(\Omega)\to \rmL^2(\partial \Omega)$ is the standard trace operator.
The dual of $\rmH^1_0(\Omega)$ is denoted by $\rmH^{-1}(\Omega)$.
Let us also define, for $\calD$ satisfying \cref{assump:D0},
\begin{equation*}
    \rmH^1_{\partial^v}(\calD) := \rmH^1(\calD) \cap \rmL^2(\calU;\rmH^1_0(\calV)).
\end{equation*}
For a function $u\in \rmL^2(\calU; \rmH^1(\calV))$, we define $\Tr_v(u) \in \rmL^2(\partial^v \calD)$ to be $(\Tr_v(u))(x,v) := \Tr(u(x, \cdot))$ whenever $(x,v) \in \partial^v \calD$.

Moreover, by $a\lesssim b$ for $a,b>0$ we mean that there is a constant $C>0$ such that $a\leq Cb$.

Finally, let $f$ be any function, we denote by $f_+:=\max\{f, 0\}$ the positive part of $f$.

\subsection{Existence and uniqueness of weak solutions}
Consider the following Dirichlet problem, where $\calD$ is as in \cref{assump:D0}:
\begin{equation}\label{eq:bvp2}
    \begin{cases}
        -\mathscr{L} u(x,v)=f(x,v),             & \textrm{ in } \calD,                       \\
        \hphantom{-\mathscr{L}}u(x,v) = g(x,v), & \textrm{ on } \partial_\mathrm{hyp} \calD, \\
    \end{cases}
\end{equation}
where $\L$ is as in \cref{eq:kfp},
$f(x,v)$ and $g(x,v)$ are functions and
\begin{equation}\label{eq:bdydata}
    g(x,v):=
    \begin{cases}
        g_1(x,v), & \textrm{ on }\partial^v \calD, \\
        g_2(x,v), & \textrm{ on }\partial^x \calD.
    \end{cases}
\end{equation}
When addressing the existence of weak solutions to the stationary equation we always require the following assumption.
\begin{assumption}\label{assump:posdiv}
    The weak divergence satisfies $\nabla_v \cdot \mathbf{b} \geq 0$.
\end{assumption}
Although this is a technical assumption, it covers many equations of interest. For instance, equation \cref{eq:kfpvillani} indeed has the equivalent form of $\mathbf{b}=0$ for its weak formulation defined with respect to $\rmL^2$ with weight $e^{-|v|^2/2 - U}$.

For the domain $\calD$ from \cref{assump:D0} we consider the Bochner spaces $\rmL^2(\calU;\rmH^1(\calV))$, as well as the weighted space $\rmL^2 (\partial^x \calD, \abs{v\cdot\mathbf{n}_x})$ and $\rmL^2(\partial^x \calD, \abs{v\cdot\mathbf{n}_x}^2)$. These are defined to be $\rmL^2$ on $\partial^x \calD$ with the measure $\abs{v\cdot\mathbf{n}_x}\rmd S$ and $\abs{v\cdot\mathbf{n}_x}^2\rmd S$ respectively, where $\rmd S$ is the restriction of the $2n-1$ dimensional Hausdorff measure onto $\partial^x \calD$.
The canonical function space associated with \cref{eq:bvp2} is the hypoelliptic space
\begin{equation*}
    \rmH^1_{\mathrm{hyp}}(\calD) := \left\lbrace u(x,v) \in  \rmL^2(\calU;\rmH^1(\calV)) \text{ s.t. } v\cdot\nabla_x u(x,v) \in L^2(\calU;\rmH^{-1}(\calV))\right\rbrace,
\end{equation*}
where $\rmH^{-1}(\calV)$ is the functional dual of $\rmH^1_0(\calV)$, endowed with the norm
\begin{equation*}
    \norm{u}_{\rmH^1_{\mathrm{hyp}}}: = \left(\norm{u}_{\rmL^2(\calU;\rmH^1(\calV))}^2 +\norm{v\cdot\nabla_x u}_{\rmL^2(\calU;\rmH^{-1}(\calV))}^2 \right)^{1/2}.
\end{equation*}
We also define the space
\begin{gather*}
    \rmH^1_{\mathrm{hyp},\partial^v }(\calD) := \left\lbrace u(x,v)\in\rmH^1_{\mathrm{hyp}}(\calD) \text{ s.t. } \Tr_v(u) = 0 \text{ on } \partial^v \calD \right\rbrace.
\end{gather*}
It is observed in~\cite{Sil22} that the classical $\rmH^1$ trace can be extended to $\rmH^1_{\textrm{hyp}}$ in a weak sense:

\begin{lemma}[Weak trace]\label{thm:weaktrace}
    Let $\calD\subset\R^{2n}$ satisfy \cref{assump:D0} where $\partial\calU$ is $\rmC^{1,1}$, then there is a linear operator
    \begin{equation}\label{eq:weaktrace}
        \tr_x: \rmH^1_\mathrm{hyp}(\calD) \to \rmL^2_{\mathrm{loc}}(\partial^x\calD, \abs{v\cdot\mathbf{n}_x}^2),
    \end{equation}
    such that, for $\varphi\in \rmC^1(\overline{\calD})$, $\varphi|_{\partial^v\calD} = 0$, we have the following integration by parts formula:
    \begin{equation}\label{eq:intbypart}
        \pair{v\cdot\nabla_x u,\, \varphi }_{\rmL^2(\calU;\rmH^{-1}(\calV)), \rmL^2(\calU;\rmH^1_0(\calV))} = \int_{\partial^x \calD} v\cdot\mathbf{n}_x \tr_x(u) \varphi\,\rmd S - \iint_{\calD} u (v\cdot\nabla_x \varphi)\,\rmd x\rmd v.
    \end{equation}
\end{lemma}

The proof of this lemma is given in \cref{sec:appendix:weaktrace}.
\begin{remark}
    Suppose that a strong trace operator exists, i.e.~$\Tr_x: \rmH^1_\mathrm{hyp}(\calD)\to\rmL^2(\partial^x\calD,\allowbreak \abs{v\cdot\mathbf{n}_x})$, then for $u, w\in\rmH^1_{\mathrm{hyp}}(\calD)$ such
    that $\Tr_v(u), \Tr_v(w)=0$ we can do integration by parts in the following sense
    \begin{equation*}
        \iint_{\calD} u (v\cdot\nabla_x w)\,\rmd x\rmd v = \int_{\partial^x\calD} v\cdot\mathbf{n}_x \Tr_x(u)\Tr_x(w)\,\rmd S - \iint_{\calD} (v\cdot\nabla_x u) w \,\rmd x\rmd v.
    \end{equation*}
    The weak trace is weak in the sense that in the integration by parts formula \cref{eq:intbypart} the function $\varphi$ cannot be taken as a function in $\rmH^1_{\mathrm{hyp}}(\calD)$.
    This means testing against the solution itself is not allowed for \cref{eq:kfp}.
    Nevertheless, there is a renormalization that behaves almost like testing the equation against the solution itself, see \cref{thm:renormal}.
\end{remark}

We first define the notion of a weak solution. The introduction of the product domain is needed to make sense of the space $u$ and $f$ belongs to.
\begin{definition}[Weak solution]\label{def:weaksol}
    Let $\Omega$ be an open set in $\R^{2n}$.
    We say that a function $u$ is a weak solution to $-\L u = f$ in $\Omega$ if, for any product domain $\calD = \calU \times \calV \subset \Omega$, we have
    $u \in \rmL^2(\calU;\rmH^1(\calV))$, $f \in \rmL^2(\calU;\rmH^{-1}(\calV))$ and
    for all $\varphi\in \rmC^1_c( \calD)$ the following holds
    \begin{equation*}
        \iint_{\Omega} \left( (\bfA \nabla_v u) \cdot \nabla_v \varphi -(\mathbf{b}\cdot\nabla_v u) \varphi+ u  (v\cdot\nabla_x \varphi) \right)\, \mathrm{d}\xi
        = \iint_{\Omega} f\varphi\, \mathrm{d}\xi.
    \end{equation*}
    If $f \equiv 0$ and the left-hand side is non-negative whenever $\varphi \geq 0$, we say that $u$ is a weak supersolution. We say that $u$ is a weak subsolution if $-u$ is a weak supersolution.
\end{definition}

Naturally, we define a weak solution to the boundary value problem \cref{eq:bvp2} as a weak solution to the equation in $\calD$ that satisfies the boundary condition in the following sense:

\begin{definition}\label{def:dirichletweaksol}
    Let $\calD\subset\R^{2n}$ satisfy \cref{assump:D0}, where $\partial \calU \in \rmC^{1,1}$. Let $g_1 \in \rmH^1_{\mathrm{hyp}} (\calD)$, $g_2 \in \rmL^2(\partial^x \calD, \abs{v\cdot\mathbf{n}_x})$ and $f\in \rmL^2(\calU;\rmH^{-1}(\calV))$.
    We say that a weak solution $u\in \rmH^1_{\mathrm{hyp}}(\calD)$ to $-\L u = f$ in $\calD$
    is a solution to the boundary value problem \cref{eq:bvp2} if $u - g_1 \in \rmH^1_{\mathrm{hyp},\partial^v}(\calD)$ and $\tr_x(u) = g_2$ on $\partial_+^x \calD$.
\end{definition}

We will present two distinct approaches to establish the existence of a weak solution.
The first one uses the Lions-Lax-Milgram theorem and works for $\calD$ satisfying \cref{assump:D0} where $\partial \calU \in \rmC^{1,1}$.
\begin{theorem}\label{thm:1}
    Let $\calD\subset\R^{2n}$, and let \cref{assump:D0,assump:1,assump:posdiv} hold, where $\partial \calU \in \rmC^{1,1}$.
    If $g_1 \in \rmH^1_{\mathrm{hyp}} (\calD)$, $g_2 \in \rmL^2(\partial^x \calD, \abs{v\cdot\mathbf{n}_x})$ and $f\in \rmL^2(\calU;\rmH^{-1}(\calV))$,
    then there exists a weak solution $u$ to \cref{eq:bvp2}
    in the sense of \cref{def:dirichletweaksol}.
\end{theorem}
For the proof, see \cref{sec:weakexist:lions}.

The other approach is via the vanishing viscosity method. This allows us to handle the case when $\partial \calU$ is $\rmC^{0,1}$.
In this case, we need a slightly different definition for weak solutions, which is weaker than the one in~\cite{Zhu22}.
\begin{definition}\label{def:weaksol:c01}
    Let $\calD\subset\R^{2n}$ satisfy \cref{assump:D0}.
    For $f \in \rmL^2(\calU;\rmH^{-1}(\calV))$, we say that $(u,u_\Gamma)$ is a \emph{weak solution pair} to $-\L u = f$ in $\calD$ if $u\in \rmL^2(\calU;\rmH^1(\calV))$ and $u_\Gamma \in \rmL^2_{\mathrm{loc}} (\partial^x \calD, \abs{v\cdot\mathbf{n}_x}^2)$ such that the Green's formula is satisfied,
    i.e.~for all $\varphi\in \rmC^1( \overline{\calD})$ with $\varphi|_{\partial^v \calD} = 0$ we have
    \begin{equation*}
        {\iint_\calD \left((\bfA \nabla_v u) \cdot \nabla_v \varphi -(\mathbf{b}\cdot\nabla_v u) \varphi+ u  (v\cdot\nabla_x \varphi) - f\varphi\right)\, \mathrm{d}x\rmd v
            = \int_{\partial^x \calD} v\cdot\mathbf{n}_x u_{\Gamma}\varphi \, \mathrm{d}S.}
    \end{equation*}
    We say that $(u,u_\Gamma)$ is a solution to the boundary value problem \cref{eq:bvp2} if $u - g_1 \in \rmL^2(\calU; \rmH^1_0(\calV))$ and  $u_{\Gamma} = g_2$ on $\partial^x_+ \calD$.
\end{definition}
It is worth noting that from our renormalization formula, \cref{thm:renormal}, when $\partial \calU$ is $\rmC^{1,1}$, that if $u$ is a weak solution in the sense of \cref{def:weaksol} in $\calD$, then $(u,\tr_x(u))$ is a weak solution pair in the sense of \cref{def:weaksol:c01}.

Following the method developed in~\cite{Zhu22}, we can construct a weak solution using the vanishing viscosity method for the stationary problem.
\begin{theorem} \label{thm:2}
    Let $\calD\subset\R^{2n}$, and let \cref{assump:D0,assump:1,assump:posdiv} hold.
    Then, for $g_1 \in \rmH^1 (\calD)$, $g_2 \in \rmL^2(\partial^x \calD, \abs{v\cdot\mathbf{n}_x})$ and $f\in \rmL^2(\calU;\rmH^{-1}(\calV))$, there exists a weak solution pair $(u, u_\Gamma)$, with $u_\Gamma\in \rmL^2(\partial^x \calD, \abs{v\cdot\mathbf{n}_x})$ such that $(u,u_\Gamma)$ is a solution to \cref{eq:bvp2}.
    Furthermore, there exists a sequence of functions $u_\varepsilon\in \rmH^1(\calD)$ (constructed through the vanishing viscosity method) that converges to the unique limit $u_\varepsilon\to u$ in $\rmL^2(\calU;\rmH^1(\calV))$ and $\Tr(u_\varepsilon)\to u_\Gamma$ in $\rmL^2(\partial^x_-\calD, \abs{v\cdot\mathbf{n}_x})$ respectively.

    If $f=0$ and $\norm{g}_{\rmL^\infty(\partial_\mathrm{hyp}\calD)}<\infty$, then the above constructed weak solution pair $(u,u_\Gamma)$ satisfies the following weak maximum principle
    \begin{equation*}
        \norm{u_{\Gamma}}_{\rmL^\infty(\partial^x_- \calD)} \leq \norm{u}_{\rmL^\infty(\calD)}\leq \norm{g}_{\rmL^\infty(\partial_\mathrm{hyp}\calD)}.
    \end{equation*}
\end{theorem}
The construction of the weak solution in this theorem can be found in \cref{thm:epslimit}. The weak maximum principle can be found in \cref{thm:weakmax}.

In order for the Perron method to work for this equation, we need the weak comparison principle which also guarantees the uniqueness of weak solutions.
\begin{theorem}[Comparison principle]\label{thm:comparison}
    Let $\calD\subset\R^{2n}$, and let \cref{assump:1,assump:D0,assump:posdiv} hold, where $\partial \calU \in \rmC^{1,1}$.
    Let $u,w$ be functions such that $-\L u,-\L w\in \rmL^2(\calU; \rmH^{-1}(\calV))$ in the sense of \cref{def:weaksol}.
    If $-\L u \leq -\L w$ in $\calD$, $\tr_x (u) \leq \tr_x (w)$ on $\partial^x_+ \calD$, and $\Tr_v (u) \leq \Tr_v (w)$ on $\partial^v \calD$,
    then $u\leq w$ a.e.~in $\calD$.
    In particular, weak solutions to \cref{eq:bvp2} in the sense of \cref{def:dirichletweaksol} are unique.

    In the case where $\partial \calU$ is $\rmC^{0,1}$, if $(u,u_\Gamma)$, $(w,w_\Gamma)$ are weak solution pairs in the sense of \cref{def:weaksol:c01}, satisfying the same conditions as above, and $u-w \in \rmL^\infty(\calD)$, then we have $u \leq w$ a.e.~in $\calD$.
\end{theorem}
For the proof, refer to \cref{sec:green}, in which we need to use the renormalization technique, see \cref{thm:renormal}, which dates back to DiPerna and Lions,~\cite{DL89}.

We note that we cannot yet exclude the case that a weak solution pair $(u,u_\Gamma)$ is unbounded when $f=0$, $u_{\Gamma}|_{\partial^x_+ \calD} = 0$ and $\partial \calU$ is only $\rmC^{0,1}$. As such we cannot directly apply \cref{thm:comparison} to obtain a uniqueness result. We leave this global boundedness problem as an open question for future work.
\begin{remark}
    The distinction between the two methods utilized in \cref{thm:1,thm:2} lies in the necessity to construct the trace function for the solution obtained through the  function space or the vanishing viscosity approximation.
    Consequently, we only establish that the trace function $\tr_x(u)$ belongs to $\rmL_{\mathrm{loc}}^2(\partial^x \calD, \abs{v\cdot\mathbf{n}_x}^2)$ due to \cref{thm:weaktrace}, whereas the trace function obtained through the vanishing viscosity method (which is constructed through the approximation) lies in $\rmL^2(\partial^x \calD, \abs{v\cdot\mathbf{n}_x})$.
    However, when $\partial \calU \in \rmC^{1,1}$, the solutions constructed in the two different ways coincide, see \cref{thm:comparison}.
    It's important to note that this difference arises due to the absence of a \emph{strong trace theorem} for the function space $\rmH_{\mathrm{hyp}}^1(\calD)$.
    It is an interesting research problem to show that the trace function $\tr_x(u)$ of a weak solution $u$ in the sense of \cref{def:weaksol} in $\calD = \calU \times \calV$ is in $\rmL^2(\partial^x \calD, \abs{v\cdot\mathbf{n}_x})$ under \cref{assump:D0}, which we currently only know for the constructed solution in \cref{thm:2}.
\end{remark}

\subsection{Perron's method} \label{sec:perron}
We consider the problem
\begin{equation}\label{eq:bvp}
    \begin{cases}
        -\mathscr{L} u(x,v)= 0, & \textrm{ in } \Omega\subset\R^{2n}, \\
        \hfill u = g,           & \textrm{ on } \partial\Omega,
    \end{cases}
\end{equation}
where $\L$ is as in \cref{eq:kfp}, satisfying \cref{assump:1,assump:posdiv}.
We begin with defining $\mathrm{K}$-harmonic and $\mathrm{K}$-superharmonic functions.
\begin{definition}\label{def:kharmonic}
    Let $\Omega\subset\R^{2n}$ be an open set.
    We say that a function $u$ is \emph{$\mathrm{K}$-harmonic} if $u$ is a continuous weak solution to $-\L u = 0$ in $\Omega$ in the sense of \cref{def:weaksol}.
\end{definition}

From \cref{thm:2,thm:reg,thm:comparison} we have the following result:
Let $\calD \subset \R^{2n}$ satisfy \cref{assump:D0}, $\partial \calU \in \rmC^{1,1}$. Then for $h\in \rmC(\partial \calD)$, there exists a unique $\mathrm{K}$-harmonic function $u \in \rmC(\overline{\calD})$ such that $u = h$ on $\partial_{\mathrm{hyp}} \calD$, satisfying $\sup_{\overline{\calD}} u \leq \sup_{\partial_{\mathrm{hyp}}\calD} h$.

The above is enough to guarantee that the notion of $\mathrm{K}$-superharmonic functions is well-defined using the comparison principle.
\begin{definition}\label{def:ksuperharmonic}
    Let $\Omega\subset\R^{2n}$ be an open set.
    A function $u: \Omega \to (-\infty,\infty]$ is called K-\emph{superharmonic} if
    \begin{enumerate}
        \item $u$ is lower semi-continuous,
        \item $u$ is finite in a dense subset of $\Omega$,
        \item $u$ satisfies the comparison principle on each product domain $\calD\ssubset \Omega$ as above: That is, for any $\mathrm{K}$-harmonic function $h \in \rmC(\overline{\calD})$ in $\calD$, such that $h\leq u$ on $\partial_{\mathrm{hyp}} \calD$, then $h\leq u$ in the whole $\calD$.
    \end{enumerate}

    A function $w$ is K-\emph{subharmonic} if $-w$ is $\mathrm{K}$-superharmonic.
\end{definition}
We can see that lower semicontinuous weak supersolutions are $\mathrm{K}$-superharmonic by \cref{thm:subisksub} making the class non-empty.

With all these notions we can define the upper and lower Perron solutions respectively.
\begin{definition}\label{def:uppersolution}
    Let $\Omega\subset\R^{2n}$ be an open set and let $g: \partial\Omega\to \R$ be a bounded function.
    A function $u$ is said to belong to the \emph{upper class} $\mathcal{U}_g$ with respect to \cref{eq:bvp} if $u$ is $\mathrm{K}$-superharmonic in $\Omega$, bounded below and satisfies
    \begin{align*}
        \liminf_{\xi \to \xi_0} u(\xi) \geq g(\xi_0)
    \end{align*}
    at each point $\xi_0 \in\partial\Omega$.
    The \emph{upper Perron solution} $\overline{H}_g$ of \cref{eq:bvp} is defined by
    \begin{align*}
        \overline{H}_g (\xi) = \inf\{ u(\xi) : u\in \mathcal{U}_g\}.
    \end{align*}
    The \emph{lower Perron solution} $\underline{H}_g$ is defined analogously.
    Finally, we say that $g$ is \emph{resolutive} with respect to $\Omega$ if $\overline{H}_g = \underline{H}_g$ in $\Omega$.
\end{definition}

\begin{theorem}\label{thm:pwbexist}
    Let $\Omega\subset\R^{2n}$ be a bounded open set, and let \cref{assump:1,assump:posdiv} hold.
    If the boundary data $g:\partial\Omega\to\R$ is bounded, then the upper and lower Perron solutions $\underline{H}_g$ and $\overline{H}_g$ to the equation \cref{eq:bvp} are $\mathrm{K}$-harmonic.
\end{theorem}
However, for the resolutivity, we need to assume slightly better properties on the coefficients of the equation, as it is difficult to construct explicit, weak supersolutions to equations in divergence form.

\begin{assumption}\label{assump:2}
    Let $\abs{\partial_{v_i} a_{ij}}<\nu$ hold for all $i,j=1,\dots,n$.
\end{assumption}
We note that this assumption formally means we can rewrite the equation in non-divergence form with a new vector field $\widetilde{\mathbf{b}} = \mathbf{b} + \nabla_v \mathbf{a}$, which is again bounded.

\begin{theorem}\label{thm:pwbresolutive}
    Let $\Omega\subset\R^{2n}$ be a bounded open set and let \cref{assump:1,assump:posdiv,assump:2} hold.
    If the boundary data $g$ is continuous on $\partial\Omega$ then $g$ is resolutive for the problem \cref{eq:bvp}.
\end{theorem}
This theorem follows from the fact that Perron solutions form a linear class, and resolutivity is thus preserved (\cref{thm:pwblinearclass}).
The complete proof is given in \cref{sec:theoremresolutive}.

\begin{definition}
    We call a boundary point $\xi_0\in\partial\Omega$ K-regular if
    \begin{equation*}
        \lim_{\xi\to\xi_0} \overline{H}_g (\xi) = g(\xi_0)
    \end{equation*}
    whenever $g:\partial\Omega\to\R$ is continuous.
\end{definition}

In \cref{sec:barrier} we provide examples of regular boundaries of certain sets that are velocity symmetric (symmetric in $v$), see \cref{exp:barrier:2}.
In particular, we can show a ball $B_r(0)\subset\R^{2n}$ is everywhere regular. This is new for the stationary Kramers-Fokker-Planck equation with general coefficients. In the case of constant coefficients, in the ultraparabolic setting, there is even a Wiener-Landis criterion, see~\cite{KLT18}.

\section{Preliminaries}\label{sec:prelim}
In this section we list some preliminary results that will be used in the proofs of the main theorems.

\subsection{Harnack's inequality and Hölder continuity}\label{sec:reg}
In this section, we recall some recent developments concerning the regularity of solutions to the kinetic Fokker-Planck equations.
We focus on the evolutionary equation as the regularity problem for the stationary problem is a straightforward consequence of the evolutionary one.

Specifically, in this section we consider the evolutionary equation
\begin{equation} \label{eq:time}
    \partial_t u - \L u = f,
\end{equation}
where $f\in \rmL^2(\calU_T;\rmH^{-1}(\calV))$, and \cref{assump:1} holds.
Throughout this section we let $\calD$ satisfy \cref{assump:D0}, and define the time cylinder
\begin{equation*}
    \calD_T := (0,T)\times\calD,\quad\partial_\mathrm{kin} \calD_T:=\left((0,T]\times\partial_\mathrm{hyp} \calD\right)\cup\left(\{0\}\times\calD\right),
\end{equation*}
Now we define a renormalized subsolution as follows.

\begin{definition}[\cite{GM22}] \label{def:rensub}
    A function $u\in \rmL^2(\calU_T;\allowbreak\rmH^1(\calV)) \cap \rmC([0,T];\rmL^2(\calD))$ is a \emph{renormalized subsolution} of $\partial_t u -\L u = f$, for $f \in \rmL^2(\calU_T;\rmH^{-1}(\calV))$, if for all $\Phi:\R \to \R$, such that $\Phi$ is a convex, non-decreasing function satisfying $\Phi' \in \rmW^{1,\infty}$, and for all $\varphi\in \rmC^1_c (\calD_T)$ satisfying $\varphi\geq0$ it holds
    \begin{multline}\label{eq:renormalsubtime*}
        \iiint_{\calD_T}-\Phi(u)\partial_t \varphi + \left(\bfA\nabla_v \Phi(u)\right)\cdot\nabla_v \varphi - \left(\mathbf{b}\cdot\nabla_v \Phi(u)\right) \varphi + \Phi(u) v\cdot\nabla_x \varphi\,\mathrm{d}z \\
        \leq
        \iiint_{\calD_T} f\Phi'(u) \varphi \,\mathrm{d}z .
    \end{multline}
\end{definition}
The family of equations \cref{eq:time} is invariant under the Galilean translation, i.e.~for $z_0,z \in \R^{2n+1}$ we define
\begin{equation}\label{eq:galilean}
    z_0\circ z = (t_0 +t, x_0 + x -t v_0, v_0 +v),
\end{equation}
and the dilation, for some $r>0$ that
\begin{equation*}
    z = (t, x, v)\to rz:= (r^2t, r^3 x, rv).
\end{equation*}
Let $Q_1 = Q_1(0):=\{-1<t <  0, \abs{x}<1, \abs{v}<1\}$ be a cube in $\R^{2n+1}$, using the invariance of the equation we define for $z_0\in \R^{1+2n}$ and $r>0$ the cube
\begin{equation*}
    Q_r(z_0):= z_0\circ rQ_1 = \{-r^2<t-t_0 < 0,\,\abs{x - x_0 + (t-t_0)v_0}< r^3,\, \abs{v-v_0}<r\},
\end{equation*}
and we write for simplicity $Q_r(0) = Q_r$.

\begin{lemma}[Harnack inequality~\cite{GM22}]\label{thm:harnacktime}
    Any non-negative renormalized weak solution $u$ to $\partial_t u - \L u = f$, satisfies the following Harnack inequality
    \begin{equation}\label{eq:harnack}
        \sup_{\widetilde{Q}^-_{r_0/4}} u \leq C(\lambda,\Lambda) \left(\inf_{Q_{r_0/4} } u + \norm{f}_{\rmL^\infty(Q_1)}\right)
    \end{equation}
    where $r_0 = 1/20$ and $\widetilde{Q}^-_{r_0/4} := Q_{r_0/4}(-19r_0^2/8, 0, 0)$.
\end{lemma}
A direct consequence is the interior H\"older regularity.
\begin{lemma}[Interior H\"older regularity~\cite{GM22}]\label{thm:holderint}
    There is an $\alpha\in (0,1)$ depending on $\lambda, \Lambda$ and $\norm{f}_{L^\infty}$ such that any renormalized weak solution $u$  of $\partial_t u - \L u = f$ in $Q_2$  satisfies
    \begin{equation*}
        \abs{u(z_1) - u(z_2)} \lesssim (1+\norm{f}_{L^\infty(Q_2)})(\norm{u}_{\rmL^2(Q_2)}+ \norm{f}_{L^\infty(Q_2)}) \abs{z_1 -z_2}^\alpha.
    \end{equation*}
\end{lemma}

\subsection{Strong maximum principle}\label{sec:strongmax}

As in for instance~\cite{AP20} we define the vector fields:
\begin{equation}\label{eq:vfield}
    V_i(t,x,v) = \partial_{v_i},\textrm{ and } Y(t,x,v) = v\cdot\nabla_x -\partial_t.
\end{equation}
\begin{definition}\label{def:admissiblecurve}
    We say that a curve $\gamma : [0, T] \to \R^{2n+1}$ is \emph{admissible} if it is absolutely continuous and
    \begin{equation*}
        \dot{\gamma}(s) = \sum_{i=1}^n \omega_i (s) V_i (\gamma(s)) + Y(\gamma(s))
    \end{equation*}
    for almost every $s\in[0,T]$ and with $\omega_1,\omega_2,\dots,\omega_n\in \rmL^1(0,T)$.
\end{definition}

We need to define the counterpart of the admissible curve and the attainable set for the stationary equation.
Given an open set $\Omega\subset\R^{2n}$, we consider the time cylinder $\Omega_{T_0} = \Omega\times(0,T_0)$ for some $T_0>0$.
For any admissible curve $\gamma\subset \Omega_{T_0}$ satisfying \cref{def:admissiblecurve} we define $\eta(s) = P(\gamma(s))\subset\Omega$ where
$P:\R^{2n+1}\to\R^{2n}$ is the projection operator $P(t,x,v) = (x, v)$.
Indeed, one can verify that if $V_i = \partial_{v_i}$ and $Y_S = v\cdot\nabla_x$ then
\begin{equation}\label{eq:stationarycurve}
    \dot{\eta}(s) = \sum_{i=1}^n \omega_i(s) V_i(\eta(s)) + Y_S(\eta(s)),
\end{equation}
where $\omega_i(s)$ are the same as in \cref{def:admissiblecurve}.
Then for $\xi_0,\xi\subset\Omega$, we say $\xi_0\prec \xi$ in $\Omega$ if there is an admissible curve $\eta\subset\Omega$ such that $\eta(0) = \xi_0$ and $\eta(T) =\xi$.
Now we define the attainable set for a point $\xi_0\in\Omega$ as follows
\begin{equation*}
    \mathcal{A}_{\xi_0} (\Omega) = \left\{
    z \in \Omega : \xi_0 \prec \xi
    \right\}.
\end{equation*}
It is easy to see that the strong maximum principle is valid in the stationary case by projection.
\begin{corollary}[Strong maximum principle~\cite{AP20}]\label{thm:strong}
    Let $\Omega\subset\R^{2n}$ be any open set.
    Let $u\geq 0$ be a weak solution to $\L u =0$ in $\Omega$.
    If $u(\xi_0)=0$ for some $\xi_0\in\Omega$,
    then $u(\xi)=0$ for all $\xi\in\overline{\mathcal{A}_{\xi_0}(\Omega)}$.
\end{corollary}

\begin{remark}\label{rmk:nonempty}
    We remark that if $\Omega$ is bounded, then $\overline{\mathcal{A}_{\xi_0}}\cap \partial\Omega \neq\emptyset$.
    For example, given any point $\xi_0=(x_0, v_0)$ we can take all $\omega_i(s)=0$ for $i=1,\dots,n$ in \cref{eq:stationarycurve}, then the curve $\gamma(s)$ can be extended to infinity as $s$ increases hence intersects with $\partial\Omega$.

    We also remark that if $\calD$ satisfies \cref{assump:D0}, then $\overline{\mathcal{A}_{\xi_0}} \cap (\partial\calD \setminus \partial_{\textrm{hyp}} \calD) = \emptyset$. This is easy to verify since if we consider an admissible curve $\gamma$ that intersects $\partial\calD$ at a point $\xi_1$, then if $\xi_1 \in \partial^x \calD$ we see that $0 \leq \dot{\gamma} \cdot \mathbf{n} = v \cdot \mathbf{n}_x$ which implies that $\xi_1 \in \partial^x_+ \calD \cap \partial^x_0 \calD$.
\end{remark}

\section{Renormalized Green's formula and the proof of \cref{thm:comparison}}\label{sec:green}

In this section we provide a constructive proof of the density of smooth functions in $\rmH^1_{\mathrm{hyp}}(\calD)$, using the convolution-translation. We provide it here as it gives us a way to understand explicitly what the weak trace is when we have a product domain such that $\partial \calU$ is $\rmC^{1,1}$. Once this is established we can prove the renormalized Green's formula and the comparison principle (\cref{thm:comparison}). We end this section with a necessary lemma concerning the boundary regularity.

First, we record the following version of the Poincaré inequality.
\begin{lemma}[Poincaré inequality]
    Let $\calD\subset\R^{2n}$ satisfy \cref{assump:D0}.
    There exists a constant $C_P = C_P (\calD)$ such that for all $u \in \rmL^2(\calU; \rmH^1_0(\calV))$ we have
    \begin{equation}\label{ineq:poincare}
        \norm{u}_{\rmL^2(\calD)} \leq C_P \norm{\nabla_v u}_{\rmL^2 (\calD)}.
    \end{equation}
\end{lemma}
\begin{proof}
    This follows easily from Fubini's theorem and the standard Poincaré inequality.

\end{proof}

Let $\rho_k(x)$ be a standard mollifier, defined as
\begin{equation*}
    \rho(x) = k^n \rho(kx)\geq 0, \textrm{ where } \rho(x)\in \rmC_c^\infty(\R^n),\, \mathrm{supp} (\rho)\subset B_1,\, \int_{\R^n} \rho(x)\,\mathrm{d}x = 1.
\end{equation*}
We introduce the \emph{convolution-translation} mollification for the domain $\calU$ and any function $u(x)\in \rmL^1({\calU})$:
\begin{equation*}
    u_{\star k}(x)=(u\star\rho_k)(x):= \int_\calU u(y)\rho_k(x - 2\mathbf{n}(x)/k -y)\,\mathrm{d}y,
\end{equation*}
where $\mathbf{n}(x)$ is defined as
\begin{equation*}
    \mathbf{n}(x): = - \frac{\nabla_x d(x,\partial\calU)}{\abs{\nabla_x d(x,\partial\calU)}} \varrho(x),
\end{equation*}
$d(x,\partial\calU)$ is the distance function, and $\varrho(x)$ is a smooth cutoff function such that for some fixed $r>0$, $\chi_{B_{r}(\partial\calU)}\leq \varrho \leq \chi_{B_{2r}(\partial\calU)}$ in $\calU$, where $B_{r}(\partial\calU):=\{x\in{\calU} \textrm{ s.t. } d(x,\partial\calU)<r\}$.
Since $\partial \calU$ is $\rmC^{1,1}$, for small enough $r$ (smaller than the radius of the interior ball condition), $\mathbf{n}(x)$ is $\rmC^{0,1}(\overline{\calU})$, and $\mathbf{n}(x) = \mathbf{n}_x$ for $x \in \partial \calU$.

\begin{lemma}[Density of smooth functions]\label{thm:contran:converge}
    Let $\calD$ satisfy \cref{assump:D0}, where $\partial\calU$ is $\rmC^{1,1}$.
    For any $u \in \rmH^1_{\mathrm{hyp}}(\calD)$, let $u_{\star k}$ be the convolution-translation of $u$ in $\calD$, then we have the following convergences
    \begin{itemize}
        \item $u_{\star k} \to u$ strongly in $\rmL^2(\calU;\rmH^1(\calV))$,
        \item $v\cdot\nabla_x u_{\star k} \to v\cdot\nabla_x u$ strongly in $\rmL^2(\calU;\rmH^{-1}(\calV))$,
    \end{itemize}
    Consequently, for each $u\in\rmH^1_{\mathrm{hyp}}(\calD)$ there exists a sequence of functions $u_\varepsilon \in \rmC^1(\calD)\cap \rmC(\overline{\calD})$ such that $u_\varepsilon \to u$ in $\rmH^1_{\mathrm{hyp}}(\calD)$ as $\varepsilon\to 0$.
\end{lemma}
\begin{proof}
    The following properties are easy to verify: if $u(x,v)\in \rmL^2(\calU;\rmH^1(\calV))$, then for fixed $v\in\calV$ and sufficiently large $k$, $u_{\star k} (\cdot, v)\in \rmC^1(\calU)\cap \rmC(\overline{\calU})$,
    $u_{\star k} \to u \in \rmL^2(\calU; \rmH^1(\calV))$ as $k\to\infty$ and if $h \in \rmL^2(\calD)$ then $h_{\star k} \to h$ in $\rmL^2(\calD)$ as $k\to\infty$.

    Next, we prove $v\cdot\nabla_x u_{\star k} \to v\cdot\nabla_x u$ strongly in $\rmL^2(\calU;\rmH^{-1}(\calV))$.
    To do this, observe
    that since $v\cdot\nabla_x u \in \rmL^2(\calU;\rmH^{-1}(\calV))$, there exists a vector with components $F^i\in\rmL^2(\calD)$, $i=0,1,\dots,n$ such that for all $\varphi\in \rmL^2(\calU;\rmH^{1}_0(\calV))$ we have
    \begin{equation}\label{eq:vnablax:repr}
        \pair{v\cdot\nabla_x u, \varphi}_{\rmL^2(\calU;\rmH^{-1}(\calV)), \rmL^2(\calU;\rmH^1_0(\calV))} = \iint_{\calD} F^0 \varphi + \sum_{i=1}^n F^i \partial_{v_i} \varphi \,\rmd x\rmd v.
    \end{equation}
    Then we compute a representation for $v\cdot\nabla_x u_{\star k}\in\rmL^2(\calU;\rmH^{-1}(\calV))$, by density of smooth functions in $\rmL^2(\calU;\rmH_0^1(\calV))$, we take $\varphi\in \rmC^\infty(\calD)\cap \rmC(\overline{\calD})$ such that $\varphi|_{\partial^v \calD} = 0$,
    \begin{align*}
        \iint_{\calD} & (v\cdot\nabla_x u_{\star k}(x,v)) \varphi(x,v) \,\rmd x\rmd v
        \\
        =             & \iint_{\calD} v\cdot\nabla_x \left(\int_{\calU} u(y,v)\rho_k(x-2\mathbf{n}(x)/k -y)\,\rmd y \right) \varphi(x,v) \,\rmd x\rmd v
        \\
        =             & \int_{\calD} \int_{\calU} u(y,v) (\mathbf{I}-\frac{2}{k}\nabla_x \mathbf{n}(x))^{\mathrm{T}} v \cdot\nabla \rho_k(x-2\mathbf{n}(x)/k -y) \varphi(x,v) \,\rmd x \rmd y\rmd  v
        \\
        =             & -\int_{\calD} \int_{\calU} u(y,v) v \cdot\nabla_y (\rho_k(x-2\mathbf{n}(x)/k -y) \varphi(x,v)) \,\rmd x \rmd y\rmd  v                                                        \\
                      & -  \frac{2}{k}\int_{\calD} \int_{\calU} u(y,v)  (\nabla_x \mathbf{n}(x))^{\mathrm{T}} v \cdot\nabla \rho_k(x-2\mathbf{n}(x)/k -y) \varphi(x,v) \,\rmd x \rmd y\rmd  v
        \\
        =             & \pair*{ v\cdot\nabla_y u(y,v),\,  \int_{\calU} \rho_k(x-2\mathbf{n}(x)/k -y)\varphi(x,v)\,\rmd x}_{\rmL^2(\calU;\rmH^{-1}(\calV)), \rmL^2(\calU;\rmH^1_0(\calV))}            \\
                      & - \iint_{\calD} R_k (x,v)\varphi(x,v) \,\rmd x\rmd v,
    \end{align*}
    where $\nabla_x \mathbf{n}(x)$ is the Jacobian of $\mathbf{n}(x)$, and
    \begin{equation*}
        R_k (x,v)= \frac{2}{k} \int_{\calU} u(y,v) (\nabla_x\mathbf{n}(x))^\mathrm{T} v\cdot\nabla \rho_k(x-2\mathbf{n}(x)/k-y)\,\rmd y.
    \end{equation*}
    Also note in the last equality we use the definition of $v\cdot\nabla_x u$ since $\int_{\calU}\rho_k(x-2\mathbf{n}(x)/k-y)\varphi(x,v)\,\rmd x \in \rmC_c^\infty(\calD)$ for $k$ sufficiently large and is a valid test function for distributions.
    It follows from \cref{eq:vnablax:repr} and the above, that
    \begin{multline*}
        \iint_{\calD} (v\cdot\nabla_x u_{\star k}(x,v)) \varphi(x,v) \,\rmd x\rmd v \\
        = \iint_{\calD} \left(F^0_{\star k}(x,v) - R_k(x,v)\right) \varphi(x,v) + \sum_{i=1}^n F^i_{\star k} \partial_{v_i} \varphi(x,v) \,\rmd x\rmd v .
    \end{multline*}
    Therefore, by linearity
    \begin{multline*}
        \pair{v\cdot\nabla_x (u_{\star k}-u),\varphi}_{\rmL^2(\calU;\rmH^{-1}(\calV)), \rmL^2(\calU;\rmH^1_0(\calV))} \\
        = \iint_{\calD}(F^0_{\star k} - F^0 - R_k) \varphi(x,v) + (F^i_{\star k}-F^i) \partial_{v_i} \varphi(x,v) \,\rmd x\rmd v
    \end{multline*}
    and from the identification of the space $\rmL^2(\calU;\rmH^{-1}(\calV))$ we have
    \begin{align*}
        \norm{v\cdot\nabla_x (u_{\star k}-u)}_{\rmL^2(\calU;\rmH^{-1}(\calV))} & \leq \norm{F^0_{\star k} - F^0 - R_k}_{\rmL^2(\calD)} + \norm{F^i_{\star k}-F^i}_{\rmL^2(\calD)}
        \\
                                                                               & \leq \norm{F^0_{\star k} - F^0 }_{\rmL^2(\calD)} + \norm{R_k}_{\rmL^2(\calD)} + \norm{F^i_{\star k}-F^i}_{\rmL^2(\calD)}.
    \end{align*}
    It suffices to show that $\norm{R_k}_{\rmL^2(\calD)}\to 0$ as $k\to \infty$.

    Indeed, if $u\in \rmH^1(\calD)$, then
    \begin{align*}
        R_k(x,v) = & \frac{2}{k}\int_{\calU} -u(y,v)  v\cdot(\nabla_x \mathbf{n}(x))^{\mathrm{T}}\nabla_y \left(\rho_k (x-2\mathbf{n}(x)/k -y)\right)\,\mathrm{d}y \\
        =          & \frac{2}{k}\int_{\calU} \left(\nabla_y u(y,v)\cdot \nabla_x \mathbf{n}(x) v\right)  \rho_k (x-2\mathbf{n}(x)/k -y)\,\mathrm{d}y               \\
        =          & \frac{2}{k} \left(\nabla_y u\right)_{\star k}(x,y) \cdot \nabla_x \mathbf{n}(x) v.
    \end{align*}
    As such $R_k \to 0$ as $k\to\infty$ in $\rmL^2({\calD})$.
    Now for any $u\in \rmL^2(\calU;\rmH^1(\calV))$, we can take a $\rmH^1(\calD)$ sequence $u_\varepsilon\to u$ as $\varepsilon\to 0$ in $\rmL^2({D})$, and find that
    \begin{multline*}
        R_k(x,v) = \underbrace{\frac{2}{k} \left(\nabla_y u_\varepsilon\right)_{\star k}(x,v)\cdot \nabla_x \mathbf{n}(x)v}_{\mathrm{I}_{k,\varepsilon}:=} \\
        + \underbrace{\frac{2}{k} \int_{\calU} (u-u_\varepsilon)(y,v) v\cdot (\nabla_x\mathbf{n}(x))^{\mathrm{T}}\nabla\rho_k (x-2\mathbf{n}(x)/k -y)\,\mathrm{d}y}_{\mathrm{II}_{k,\varepsilon}:=} .
    \end{multline*}
    It is clear that for fixed $\varepsilon$, the first term $\mathrm{I}_{k,\varepsilon}\to 0$ in $\rmL^2({\calD})$ as $k\to 0$.
    For the second term, we compute
    \begin{align*}
         & \norm{\mathrm{II}_{k,\varepsilon}}^2_{\rmL^2(\calD)}                                                                                                                                                             \\
         & \leq 4 \norm{v\nabla_x\mathbf{n}(x)}^2_{\rmL^\infty(\calD)} \iint_\calD \left(\int_\calU \abs{(u-u_\varepsilon)(y,v)}\abs{(k^n \nabla \rho)(kx-2\mathbf{n}(x)-ky)}\,\mathrm{d}y\right)^2\,\mathrm{d}x\mathrm{d}v \\
         & \leq 4 \norm{v\nabla_x\mathbf{n}(x)}^2_{\rmL^\infty(\calD)}
        \norm{u-u_\varepsilon}_{\rmL^2(\calD)}^2 \norm{k^n \nabla\rho(ky)}^2_{\rmL^1(\calD)}
    \end{align*}
    where in the second inequality we use Young's convolution inequality.
    Note that for all $k\geq 1$,
    \begin{equation*}
        \int_{\R^n} \abs{k^n \nabla \rho(ky)}\,\mathrm{d}y = C
    \end{equation*}
    for some constant $C$ independent of $k$.
    Hence, by we have $\lim_{k\to\infty} \norm{\mathrm{II}_{k,\varepsilon}}_{\rmL^2(\calD)}^2 \leq C' \norm{u-u_\varepsilon}_{\rmL^2(\calD)}^2$ in $\rmL^2(\calD)$ for each fixed $\varepsilon>0$ and some other constant $C'$ independent of $\varepsilon$.
    By combining the convergence of $\mathrm{I}_{k,\varepsilon}$ and $\mathrm{II}_{k,\varepsilon}$
    we prove that $\norm{R_k}_{\rmL^2(\calD)}\to 0$ as $k\to \infty$.

    For the density of smooth function, we use the fact that each $u_{\star k}$ can be approximated by $\rmC^1(\calD)\cap \rmC(\overline{D})$ functions in $\rmH^1(\calD)$ norm.
    By combining both approximations using a diagonalization argument, the lemma is proved.

\end{proof}

\subsection{Renormalized Green's formula}

With the density of smooth functions in $\rmH^1_{\mathrm{hyp}}(\calD)$ at hand we can now establish the renormalized Green's formula.
We note that similar results and techniques can be found in the literature, see for instance~\cite{Mis00,Zhu22}. We would however like to point out that our definition of weak solution (\cref{def:weaksol}) is weaker than the one in~\cite{Zhu22}, in that we can use the weak trace and the density of smooth functions to establish the renormalized Green's formula.
The proof can be found in \cref{sec:appendix}.
\begin{lemma}\label{thm:renormal}
    Let $\calD \subset \R^{2n}$, and let \cref{assump:1,assump:D0} hold, where $\partial \calU \in \rmC^{1,1}$. Let $f \in \rmL^2(\calU;\rmH^{-1}(\calV))$, if $u \in \rmH^1_{\mathrm{hyp}}(\calD)$ is a weak solution in the sense of \cref{def:weaksol},
    then for any $\Phi:\R\to\R$ such that $\Phi'\in \rmW^{1,\infty}(\R)$ and either
    \begin{enumerate}
        \item \label{item:renormal1} $\Phi'(\Tr_v(u)) = 0$ a.e.~on $\partial^v \calD$ and $\varphi\in \rmC^1 (\overline{\calD})$, or
        \item \label{item:renormal2} $\varphi\in \rmC^1(\overline{\calD})$ with $\varphi|_{\partial^v \calD} = 0$,
    \end{enumerate}
    we have
    \begin{multline}\label{eq:renormal}
        \iint_{\calD} \left(\bfA\nabla_v \Phi(u)\right)\cdot \nabla_v \varphi + \varphi \Phi''(u)(\bfA\nabla_v u)\cdot\nabla_v u  - \left(\mathbf{b}\cdot\nabla_v \Phi(u)\right) \varphi
        + \Phi(u) v\cdot\nabla_x \varphi\,\mathrm{d}\xi \\
        = \iint_{\calD} f\Phi'(u) \varphi \,\mathrm{d}\xi + \int_{\partial^x \calD} v\cdot\mathbf{n}_x \Phi(\tr_x(u))\varphi\,\mathrm{d}S.
    \end{multline}
    In the special case that $\Phi(u) = u$ we get that $(u,\tr_x(u))$ is a weak solution pair in the sense of \cref{def:weaksol:c01}.
\end{lemma}

\begin{remark}
    When $\partial \calU$ is only $\rmC^{0,1}$, we can still establish the renormalized Green's formula when we have a priori that $u \in \rmL^\infty(\calD)$.
    Specifically, the proof of \cref{thm:renormal} gives the existence of $u_{\Gamma} \in \rmL^\infty(\partial^x \calD)$ such that $(u,u_{\Gamma})$ is a weak solution pair in the sense of \cref{def:weaksol:c01}, satisfying
    \begin{equation*}
        \norm{u_{\Gamma}}_{\rmL^\infty(\partial^x \calD )} \leq \norm{u}_{\rmL^\infty(\calD)}.
    \end{equation*}
    Furthermore, if in addition $\Phi \in \rmW^{2,\infty}_{\mathrm{loc}}$ and either of \cref{item:renormal1,item:renormal2} holds, then \cref{eq:renormal} holds with $\tr_x(u)$ replaced by $u_{\Gamma}$.
\end{remark}

\begin{remark}
    If the source term $f\in\rmL^2(\calD)$ then $\Phi$ can be chosen to be convex and such that $\Phi'\in \rmL^\infty(\R)$. We then obtain that $\Phi(u)$ is a renormalized subsolution in the sense of satisfying the following inequality with $\varphi \geq 0$:
    \begin{multline*}
        \int_{\calD} \left(\bfA\nabla_v \Phi(u)\right)\cdot \nabla_v \varphi - \left(\mathbf{b}\cdot\nabla_v \Phi(u)\right) \varphi
        + \Phi(u) v\cdot\nabla_x \varphi\,\mathrm{d}\xi
        \\
        \leq \int_{\calD} f\Phi'(u) \varphi \,\mathrm{d}\xi + \int_{\partial^x \calD} v\cdot\mathbf{n}_x \Phi(\tr_x(u))\varphi\,\mathrm{d}S.
    \end{multline*}
    That is, we can choose $\Phi(u) = u_+$.
\end{remark}

A consequence of the renormalized Green's formula \cref{eq:renormal} is the weak maximum principle for a weak solution.
Recall that we are under \cref{assump:1,assump:D0,assump:posdiv}, which is sufficient for the weak maximum principle to hold.

\subsection{Proof of \cref{thm:comparison}}

Let $h = u-w$, then $-\L h = f \leq 0$, $\tr_x(h) = \tr_x(u)-\tr_x(w) \leq 0$ on $\partial^x_+ \calD$, and $\Tr_v h \leq 0$ on $\partial^v \calD$.
Consider the following function $\Psi_k$ for $k>0$:
\begin{equation*}
    \Psi_k(h) =
    \begin{cases}
        0 & h \leq 0  \\
        h & 0 < h < k \\
        k & h \geq k
    \end{cases}
\end{equation*}
We define $\Phi_k(h) = \int_{-\infty}^{h} \Psi_k(s)\,\mathrm{d}s$.

It is clear by definition that $\Phi_k' = \Psi_k \in \rmW^{1,\infty}(\R)$ and $\Phi_k(r) = O(r)$ as $r \to \infty$.
Since $\Phi_k(h) = \Phi_k'(0)=0$ when $h \leq 0$, we have that $\Phi_k'(\Tr_v h) = 0$ on $\partial^v \calD$.
Now applying \cref{thm:renormal} with $\Phi_k$ and $\varphi=1$, we get, since $f \leq 0$,
\begin{equation*}
    \int_{\calD} \Phi_k''(h)(\bfA \nabla_v h)\cdot\nabla_v h  - \left(\mathbf{b}\cdot\nabla_v \Phi_k(h)\right)\,\mathrm{d}\xi
    \leq \int_{\partial^x \calD} v\cdot\mathbf{n}_x \Phi_k(\tr_x(h))\,\mathrm{d}S.
\end{equation*}
By \cref{assump:posdiv}, we can integrate by parts to get
\begin{equation*}
    -\int_\calD \mathbf{b} \cdot \nabla_v \Phi_k(h) \,\mathrm{d}\xi
    =
    \underset{=0}{\underbrace{-\int_{\calD} \nabla_v \cdot (\mathbf{b} \Phi_k(h)) \,\mathrm{d}\xi}} + \int_{\calD} (\nabla_v \cdot \mathbf{b}) \Phi_k(h) \,\mathrm{d}\xi
    \geq 0,
\end{equation*}
where the first term on the right-hand side is zero since $\Phi_k(\Tr_v h) =0$ on $\partial^v \calD$.

By our assumption, we have that $\Phi_k(\tr_x(h)) = 0$ on $\partial^x_+ \calD$.
Hence,
\begin{align*}
    \int_{\calD} \Phi_k''(h)|\nabla_v h|^2 \,\mathrm{d}\xi + \int_{\partial^x_- \calD} |v\cdot\mathbf{n}_x| \Phi_k(\tr_x(h))\,\mathrm{d}S \leq 0.
\end{align*}
Since this holds for all $k$, we have that $h \leq 0$ a.e.~in $\calD$ and $\tr_x(h) \leq 0$ a.e.~on $\partial^x_- \calD$.
This concludes the proof.
\begin{flushright}
    \qed
\end{flushright}

Now the weak maximum principle follows as a corollary as constants solve $-\L C = 0$ in the sense of \cref{def:weaksol}.

\begin{corollary}\label{thm:weakunique}
    Let $\calD \subset \R^{2n}$, and let \cref{assump:1,assump:D0,assump:posdiv} hold, where $\partial \calU \in \rmC^{1,1}$.
    Let $f = 0$‚ $g_1\in \rmL^\infty(\partial^v \calD)$ and $g_2\in \rmL^\infty(\partial^x \calD)$.
    If $u$ is a weak solution to \cref{eq:bvp2} in the sense of \cref{def:dirichletweaksol}, then it holds that
    \begin{equation*}
        \sup_{\partial^x_- \calD} \tr_x(u) \leq \sup_{\partial_{\mathrm{hyp}}\calD} g_+ \textrm{ and } \sup_{\calD} u \leq \sup_{\partial_{\mathrm{hyp}}\calD} g_+.
    \end{equation*}
    In particular,
    \begin{equation*}
        \sup_{\partial^x_- \calD} \abs{\tr_x(u)} \leq \sup_{\partial_{\mathrm{hyp}}\calD} \abs{g} \textrm{ and } \sup_{\calD} \abs{u}\leq \sup_{\partial_{\mathrm{hyp}}\calD} \abs{g}.
    \end{equation*}

    If $\partial \calU$ is only $\rmC^{0,1}$, then the same result holds provided the weak solution $u$ is bounded.
\end{corollary}

\subsection{Boundary regularity}
A consequence of the regularity of the time-dependent equation, see~\cite{Zhu22} and the convergence of the trace of the convolution-translation (see \cref{thm:contran:converge}), is the boundary regularity of the weak solution.

\begin{lemma} \label{thm:reg}
    Let $\calD\subset\R^{2n}$, and let \cref{assump:1,assump:D0} hold, where $\partial \calU \in \rmC^{1,1}$.
    Let $\norm{f}_{\rmL^\infty(\calD)}<\infty$.
    Suppose $u\in\rmH^1_{\mathrm{hyp}}(\calD)$ is a bounded weak solution to \cref{eq:bvp2} in the sense of \cref{def:dirichletweaksol}.
    If $g\in \rmC(\partial_{\mathrm{hyp}} \calD)$ then $u\in \rmC(\overline{\calD})$ and $u = \tr_x(u)$ on $\partial^x \calD$.
    Furthermore, if $g\in \rmC^\alpha(\partial_{\mathrm{hyp}} \calD)$ for some $0<\alpha<1$, then $u\in \rmC^\beta(\overline{\calD})$ for some $0<\beta\leq \alpha$.
\end{lemma}
\begin{proof}
    As the stationary solution is simply a constant in time solution to the evolutionary equation $\partial_t u - \L u = f$, the boundary regularity follows from~\cite{Zhu22} and the proof is omitted.

    To prove that $u = \tr_x(u)$, it suffices to show that $u_{\star k}\to u$ in $\rmL^2_{\mathrm{loc}}(\partial^x\calD, \abs{v\cdot\mathbf{n}_x}^2)$, where $u_{\star k}$ is the convolution-translation of $u$ in $\calD$ defined in the beginning of this section.
    Indeed,
    \begin{equation*}
        \iint_{\partial^x\calD} \abs{u_{\star k} - u}^2\,\mathrm{d}S(x,v)
        \leq \iint_{\partial^x\calD}  \sup_{y\in B_{1/k}(x - 2\mathbf{n}(x)/k)} C\abs{y - x}^\beta \,\mathrm{d}S(x,v),
    \end{equation*}
    where $C$ is a constant and $\beta$ is the H\"older exponent and both are independent of $k$ (the term $\abs{y-x}^\beta$ can be replaced by other modulus of continuity).
    By sending $k\to\infty$ we get the desired result.

\end{proof}

\section{Existence of weak solutions}
\subsection{Lions-Lax-Milgram theorem}\label{sec:weakexist:lions}
In this section, we present the Lions-Lax-Milgram approach.
Indeed, we are motivated by~\cite{AAMN21}, see also~\cite{FSGM23}.
Again, note that we need to assume that $\mathbf{b}$ satisfies \cref{assump:posdiv}.
\begin{proof}[Proof of \cref{thm:1}]
    W.L.O.G.~assume that $g_1 = 0$.
    Consider the bilinear form
    \begin{equation*}
        E(u,\psi) := \iint_\calD \left(\bfA\nabla_v u\right)\cdot\nabla_v \psi - (\mathbf{b}\cdot\nabla_v u)\psi + u (v\cdot\nabla_x\psi)\,\mathrm{d}\xi,
    \end{equation*}
    for all $u\in \rmL^2(\calU;\rmH^1_0 (\calV))$ and all
    $\psi\in \Upsilon:=\{\psi \in \rmH^1(\calD)\cap \rmL^2(\calU;\rmH^1_0(\calV)) \textrm{ s.t. } \Tr(\psi)|_{\partial^x_-\calD} = 0\}$.
    The space of test functions $\Upsilon$ is equipped with the norm
    \begin{equation*}
        \norm{\psi}_\Upsilon^2 := \iint_\calD (\psi^2 +\abs{\nabla_v \psi}^2) \,\mathrm{d}\xi + \iint_{\partial^x_+ \calD} v\cdot\mathbf{n}_x \Tr(\psi)^2\,\mathrm{d}S.
    \end{equation*}
    It is obvious that $\Upsilon$ is continuously embedded in $\rmL^2(\calU;\rmH^1_0 (\calV))$, namely
    $\norm{\psi}^2_{\rmL^2(\calU;\rmH^1_0 (\calV))} \leq \norm{\psi}_\Upsilon^2$.
    Next, for $\psi\in \Upsilon$, using \cref{assump:posdiv} it holds that
    \begin{equation*}
        E (\psi,\psi) = \iint_{\calD} \left(\bfA\nabla_v \psi\right)\cdot\nabla_v \psi - (\mathbf{b}\cdot\nabla_v \psi)\psi \,\mathrm{d}\xi + \frac{1}{2}\int_{\partial^x_+\calD} v\cdot\mathbf{n}_x \mathrm{Tr}(\psi)^2 \,\mathrm{d}S
        \geq C \norm{\psi}_\Upsilon^2 .
    \end{equation*}
    Therefore, the bilinear form $E$ is $\Upsilon$-coercive.
    Now we define a functional $F\in \Upsilon^*$ as follows
    \begin{equation*}
        F(\psi) := \iint_{\calD} f \psi\,\mathrm{d}\xi + \int_{\partial^x_+\calD} v\cdot\mathbf{n}_x g_2 \Tr(\psi)\,\mathrm{d}S.
    \end{equation*}
    It is clear that $F \in \Upsilon^\ast$ since by duality and Cauchy-Schwartz
    \begin{align*}
        |F(\psi)|
        \leq &
        \|f\|_{\rmL^2(\calU; \rmH^{-1}(\calV))}\|\psi\|_{\rmL^2(\calU; \rmH^1(\calV))} + \|g_2\|_{\rmL^2(\partial^x_+ \calD; \abs{v \cdot \mathbf{n}_x})} \|\Tr(\psi)\|_{\rmL^2(\partial^x_+ \calD; \abs{v \cdot \mathbf{n}_x})}
        \\
        \leq &
        C \norm{\psi}_\Upsilon.
    \end{align*}
    It then follows from the Lions-Lax-Milgram theorem, see~\cite[Theorem 2.1]{Sho97}, that for any $F\in \Upsilon^*$ there is a solution $u\in \rmL^2(\calU;\rmH^1_0(\calV))$ to the problem $E(u,\psi) = F(\psi)$ for all $\psi \in \Upsilon$.

    Now by \cref{thm:weaktrace} we have that there exists a trace function $\tr_x(u)$ such that the renormalization formula \cref{eq:renormal} holds. As such, $\tr_x(u) |_{\partial^x_+\calD} = g_2$ and the proof is complete.

\end{proof}

\subsection{Vanishing viscosity approximation}\label{sec:weakexist:vs}
In this section, we present the vanishing viscosity method coupled with a Robin boundary condition similar to~\cite{Zhu22}.

Throughout this section, we let \cref{assump:1,assump:D0,assump:posdiv} hold.

Let's consider the following problem:
\begin{equation}\label{eq:bvp3}
    \left\{
    \begin{array}{rll}
        -\mathscr{L}^\varepsilon u
         & = f,                          & \textrm{in } \mathcal{D},            \\
        u
         & = g_1,                        & \textrm{on } \partial^v \mathcal{D}, \\
        \varepsilon \frac{\partial u}{\partial \mathbf{n}_x} + (v\cdot\mathbf{n}_x)_+ u
         & = (v\cdot\mathbf{n}_x)_+ g_2, & \textrm{on } \partial^x \mathcal{D},
    \end{array}\right.
\end{equation}
where $\frac{\partial u}{\partial \mathbf{n}_x}:= \mathbf{n}_x\cdot\nabla_x u$, and $g_1$ and $g_2$ represent boundary data as in \cref{eq:bdydata}.
The operator $\mathscr{L}^\varepsilon$ is defined as:
\begin{equation}\label{eq:opeps}
    \mathscr{L}^\varepsilon := \mathscr{L} + \varepsilon \Delta_x = \nabla_\xi\cdot(\widehat{\mathbf{A}}_\varepsilon \nabla_\xi) + \widehat{\mathbf{b}}\cdot\nabla_\xi ,
\end{equation}
where
\begin{equation*}
    \widehat{\mathbf{A}}_\varepsilon:= \begin{pmatrix}
        \varepsilon \mathrm{I} & 0          \\
        0                      & \mathbf{A}
    \end{pmatrix}
    \textrm{ and }
    \widehat{\mathbf{b}}:= \begin{pmatrix}
        v \\
        \mathbf{b}
    \end{pmatrix}.
\end{equation*}

We remark that the Robin boundary condition is canonical from \cref{eq:opeps} in the sense that if we test \cref{eq:opeps} with the solution itself and integrate by parts then we can see the contribution of Robin boundary terms.

Now, let's define weak solutions to \cref{eq:bvp3}:
\begin{definition}
    Let $f\in\rmL^2(\calU;\rmH^{-1}(\calV))$, $g_1\in \rmH^1(\calD)$ and $g_2\in \rmL^2(\partial^x \calD,\abs{v\cdot\mathbf{n}_x})$.
    We say that $u\in \rmH^1(\calD)$ is a weak solution to \cref{eq:bvp3} if $u-g_1 \in \rmH^1_{\partial^v}(\calD)$ and for all $\varphi\in \rmH^1_{\partial^v}(\calD)$, the following holds:
    \begin{multline}\label{def:epsweak}
        \iint_\calD \left(\widehat{\bfA}_\varepsilon \nabla_\xi u\right)\cdot \nabla_\xi \varphi - (\widehat{\mathbf{b}}\cdot \nabla_\xi u) \varphi\,\mathrm{d}\xi + \int_{\partial^x \calD} (v\cdot\mathbf{n}_x)_+ \Tr (u) \varphi \,\mathrm{d}S \\
        = \iint_\calD f\varphi\,\mathrm{d}\xi + \int_{\partial^x \calD} (v\cdot\mathbf{n}_x)_+ g_2 \varphi\,\mathrm{d}S,
    \end{multline}
    where $\mathrm{d}S$ is the standard surface measure on $\partial \calD$.
\end{definition}
\begin{lemma}\label{thm:epsmax}
    For each fixed $\varepsilon$, if $u_\varepsilon$ is a weak solution to \cref{eq:bvp3} with $f=0$ then
    \begin{equation*}
        \sup_{\partial^x \calD} u_\varepsilon \leq \sup_\calD u_\varepsilon \leq \sup_{\partial_\mathrm{hyp} \calD} g_+
    \end{equation*}
\end{lemma}
\begin{proof}
    The first inequality is trivial.

    For the second inequality, we follow the proof of~\cite[Theorem 8.1]{GT01}.
    Let's take $\varphi = (u - M-k )_+$ where $M=\sup_{\partial_\mathrm{hyp} D} g_+$ and $0\leq k <\sup_{\calD} u-M$ (if no such $k$ exists we are done).
    It follows that $\varphi\in \rmH^1_{\partial^v}(\calD)$ and in particular $\varphi\geq 0$, hence
    \begin{equation*}
        \iint_D \left(\widehat{\bfA}_\varepsilon \nabla_\xi u \right)\cdot \nabla_\xi \varphi - (\widehat{\mathbf{b}}\cdot \nabla_\xi u) \varphi\,\mathrm{d}z \leq  \int_{\partial^x D} (v\cdot\mathbf{n}_x)_+ (g_2 - \Tr(u)) \varphi\,\mathrm{d}S \leq 0,
    \end{equation*}
    which is equivalent to
    \begin{equation*}
        \iint_D \left(\widehat{\bfA}_\varepsilon \nabla_\xi \varphi \right)\cdot \nabla_\xi \varphi \,\mathrm{d}z \leq \int_{\Sigma} (\widehat{\mathbf{b}}\cdot \nabla_\xi \varphi) \varphi\,\mathrm{d}\xi
    \end{equation*}
    where $\Sigma = \mathrm{supp}(\nabla_\xi \varphi)\subset\mathrm{supp}(\varphi)$.
    Now this is the same energy inequality as in the proof of~\cite[Theorem 8.1]{GT01}, and we can proceed in the same way to conclude the lemma.

\end{proof}
\begin{proposition}\label{thm:epstime:exist}
    Let $\calD\subset\R^{2n}$, and let \cref{assump:1,assump:D0,assump:posdiv} hold.
    If $f\in \rmL^2(\calU;\rmH^{-1}(\calV))$, $g_1\in \rmH^1(\calD)$ and $g_2\in \rmL^2(\partial^x \calD,\abs{v\cdot\mathbf{n}_x})$, then for each fixed $\varepsilon$ there exists a weak solution $u_\varepsilon\in \rmH^1(\calD)$ to the Dirichlet problem \cref{eq:bvp3}.
\end{proposition}
\begin{proof}
    This proof is an application of the Lax-Milgram theorem, just to note that now we have the bilinear form
    \begin{equation*}
        E(u, \varphi):= \iint_\calD \left(\widehat{\bfA}_\varepsilon \nabla_\xi u\right)\cdot \nabla_\xi \varphi  - (\widehat{\mathbf{b}}\cdot \nabla_\xi u) \varphi\,\mathrm{d}\xi+ \int_{\partial^x \calD} (v\cdot\mathbf{n}_x)_+ \Tr (u) \varphi \,\mathrm{d}S,
    \end{equation*}
    and a linear functional
    \begin{equation*}
        F(\varphi):= \iint_D f \varphi \,\mathrm{d}\xi + \int_{\partial^x D} (v\cdot\mathbf{n}_x)_+ g_2 \varphi\,\mathrm{d}S.
    \end{equation*}
    The rest of the proof is similar to that of~\cite[Theorem 8.3]{GT01}.

\end{proof}
\begin{lemma}\label{thm:epsenergy}
    Let $\calD \subset \R^{2n}$, and let \cref{assump:1,assump:D0,assump:posdiv} hold.
    Let $g_1 = 0$ and $g_2\in \rmL^2(\partial^x \calD, \abs{v\cdot\mathbf{n}_x})$.
    Then for each fixed $\varepsilon$ the weak solution $u_\varepsilon$ to \cref{eq:bvp3} satisfies the following energy estimate
    \begin{multline*}
        \iint_{\calD} \left(\varepsilon \abs{\nabla_x u_\varepsilon}^2 + \frac{\lambda}{4}\abs{\nabla_v u_\varepsilon}^2 + \frac{\lambda}{8C_P}\abs{u_\varepsilon}^2\right)\,\mathrm{d}\xi
        + \frac{1}{4}\int_{\partial^x \calD} \abs{v\cdot\mathbf{n}_x} u_\varepsilon^2\,\mathrm{d}S \\
        \leq \frac{2C_P}{\lambda} \iint_{\calD} f^2 \mathrm{d}\xi + \int_{\partial^x \calD} (v\cdot\mathrm{n}_x)_+ g_2^2\,\mathrm{d}S,
    \end{multline*}
    where $\lambda$ is from \cref{assump:1}.
\end{lemma}
\begin{proof}
    It follows immediately by taking $\varphi = u_\varepsilon$ in \cref{def:epsweak} and using \cref{assump:1} and \cref{ineq:poincare}.

\end{proof}
By letting $\varepsilon\to 0$ we obtain the following existence result for \cref{eq:bvp2} for Lipschitz domains.
\begin{proposition}\label{thm:epslimit}
    Let $\calD\subset\R^{2n}$, and let \cref{assump:1,assump:D0,assump:posdiv} hold.
    If $g_1 \in \rmH^1 (\calD)$, $g_2 \in \rmL^2(\partial^x D, \abs{v\cdot\mathbf{n}_x})$ and $f\in \rmL^2(\calU;\rmH^{-1}(\calV))$,
    then there exists a pair of functions $u\in \rmL^2(\calU;\rmH^1(\calV))$ and $u_{\Gamma}\in\rmL^2(\partial^x \calD, \abs{v\cdot\mathbf{n}_x})$ such that $u_\varepsilon\to u$ in $\rmL^2(\calU;\rmH^1(\calV))$ and $\Tr(u_\varepsilon)\to u_{\Gamma}$ in $\rmL^2(\partial^x_- \calD, \abs{v\cdot\mathbf{n}_x})$.
    Furthermore, $(u,u_\Gamma)$ is a weak solution pair to \cref{eq:bvp2} in the sense of \cref{def:weaksol:c01}.
\end{proposition}
\begin{proof}
    Following the lines of~\cite{Zhu22} adapted to our setting.
    W.L.O.G.~we assume that $g_1 = 0$.
    Let $\varepsilon_k = \frac{1}{k^2}$, and we do the energy estimates for $\mathscr{L}^{\varepsilon_k} u_{\varepsilon_k} - \mathscr{L}^{\varepsilon_{k+1}} u_{\varepsilon_{k+1}}$.
    Define $h_k := u_{\varepsilon_k}-u_{\varepsilon_{k+1}}$, for all $\varphi\in \rmH^1_{\partial^v }(\calD)$ it follows that
    \begin{multline*}
        \iint_\calD (\bfA \nabla_v h_k)\cdot\nabla_v \varphi +(\varepsilon_k + \varepsilon_{k+1}) \nabla_x h_k \cdot\nabla_x\varphi -  (\mathbf{b}\cdot\nabla_v h_k) + (v\cdot\nabla_x h_k)\varphi\,\mathrm{d}\xi \\
        + \int_{\partial^x \calD} (v\cdot\mathbf{n}_x)_+ h_k\varphi\,\mathrm{d}S=\iint_\calD (\varepsilon_{k+1} \nabla_x u_{\varepsilon_{k}} -\varepsilon_k\nabla_x u_{\varepsilon_{k+1}})\cdot\nabla_x \varphi\,\mathrm{d}\xi .
    \end{multline*}
    This is an equation of the same structure as \cref{eq:bvp3} hence we proceed similarly as in the proof of \cref{thm:epsenergy} to obtain
    \begin{align*}
        \begin{split}
            \iint_\calD (\varepsilon_k +\varepsilon_{k+1})\abs{\nabla_x h_k}^2 +\frac{\lambda}{4}\abs{\nabla_v h_k}^2 + \frac{\lambda}{4C_P} \abs{h_k}^2 \,\mathrm{d}\xi + \frac{1}{2}\int_{\partial^x \calD} \abs{v\cdot\mathbf{n}_x} h_k^2 \,\mathrm{d}S \\
            \leq \iint_\calD (\varepsilon_{k+1} \nabla_x u_{\varepsilon_{k}} -\varepsilon_k\nabla_x u_{\varepsilon_{k+1}})\cdot\nabla_x h_k \,\mathrm{d}\xi \Rightarrow
        \end{split} \\
        \begin{split}
            \iint_\calD \abs{\sqrt{\varepsilon_k}\nabla_x u_{\varepsilon_k}-\sqrt{\varepsilon_{k+1}}\nabla_x u_{\varepsilon_{k+1}}}^2 +\frac{\lambda}{4}\abs{\nabla_v h_k}^2 + \frac{\lambda}{4C_P} \abs{h_k}^2 \,\mathrm{d}\xi + \frac{1}{2}\int_{\partial^x \calD} \abs{v\cdot\mathbf{n}_x} h_k^2 \,\mathrm{d}S \\
            \leq \frac{(\sqrt{\varepsilon_k}-\sqrt{\varepsilon_{k+1}})^2}{\sqrt{\varepsilon_k \varepsilon_{k+1}}}\iint_\calD \varepsilon_k \abs{\nabla_x u_{\varepsilon_k}}^2 + \varepsilon_{k+1}\abs{\nabla_x u_{\varepsilon_{k+1}}}^2\,\mathrm{d}\xi
        \end{split}
    \end{align*}
    Observe that $\frac{(\sqrt{\varepsilon_k}-\sqrt{\varepsilon_{k+1}})^2}{\sqrt{\varepsilon_k \varepsilon_{k+1}}} = \frac{1}{k(k+1)}\to 0$ as $k\to\infty$
    and $\int_\calD \varepsilon_k \abs{\nabla_x u_{\varepsilon_k}}^2 \,\mathrm{d}\xi$ is uniformly bounded for all $\varepsilon_k$ due to \cref{thm:epsenergy}.
    Hence, $h_k$ is Cauchy in $\rmL^2(\calU; \rmH^1(\calV))$, $\Tr(h_k)$ is Cauchy in $\rmL^2(\partial^x \calD, \abs{v\cdot\mathbf{n}_x})$ and $\sqrt{\varepsilon_k}\nabla_x u_{\varepsilon_k}$ is Cauchy in $\rmL^2(\calD)$,
    which implies that there is a function $u\in \rmL^2(\calU; \rmH^1(\calV))$ and a function $\widehat{u}_{\Gamma}\in \rmL^2(\partial^x \calD; |v \cdot n_x|)$ such that
    \begin{gather}
        \nonumber u_{\varepsilon} \to u,\textrm{ and } \nabla_v u_\varepsilon \to \nabla_v u,\textrm{ in } \rmL^2 (\calD),\\
        {\varepsilon}\nabla_x u_\varepsilon \to 0 \textrm{ in } \rmL^2(\calD),\label{eq:epsconverge}\\
        \nonumber \Tr(u_\varepsilon) \to \widehat u_{\Gamma} \textrm{ in } \rmL^2 (\partial^x \calD, \abs{v\cdot\mathbf{n}_x}).
    \end{gather}
    We now use integration by parts on the $v\cdot\nabla_x $ term to rewrite \cref{def:epsweak} in the following form
    \begin{multline}\label{def:epsweak2}
        \iint_\calD (\bfA\nabla_v u_\varepsilon)\cdot\nabla_v \varphi + \varepsilon \nabla_x u_\varepsilon \nabla_x \varphi - (\mathbf{b}\cdot\nabla_v u_\varepsilon) \varphi + u_\varepsilon(v\cdot\nabla_x \varphi)\,\mathrm{d}\xi  \\
        = \iint_\calD f\varphi\,\mathrm{d}\xi + \int_{\partial^x_+ \calD} v\cdot\mathbf{n}_x g_2\varphi\,\mathrm{d}S + \int_{\partial^x_- \calD} v\cdot\mathbf{n}_x u_\varepsilon \varphi\,\mathrm{d}S.
    \end{multline}
    By sending $\varepsilon\to 0$ we get from \cref{eq:epsconverge} that
    \begin{multline*}
        \iint_\calD (\bfA\nabla_v u)\cdot\nabla_v \varphi - (\mathbf{b}\cdot\nabla_v u) \varphi + u(v\cdot\nabla_x \varphi)\,\mathrm{d}\xi  \\
        = \iint_\calD f\varphi\,\mathrm{d}\xi + \int_{\partial^x_+ \calD} v\cdot\mathbf{n}_x g_2\varphi\,\mathrm{d}S + \int_{\partial^x_- \calD} v\cdot\mathbf{n}_x \widehat u_{\Gamma} \varphi\,\mathrm{d}S.
    \end{multline*}

    Finally, we define
    \begin{equation*}
        u_\Gamma :=
        \begin{cases}
            g_2,                 & \textrm{ on } \partial^x_+\calD, \\
            \widehat u_{\Gamma}, & \textrm{ on } \partial^x_-\calD,
        \end{cases}
    \end{equation*}
    and the proof is complete.

\end{proof}
\begin{remark}
    We remark that in the proof $\lim_{\varepsilon\to 0} \Tr(u_\varepsilon)$ does not necessarily equal $g_2$ on $\partial^x_+\calD$ because from the proof we cannot exclude the case that $\mathbf{n}_x\cdot\nabla_x u_\varepsilon$ blows up when $\varepsilon\to0$.
\end{remark}
\begin{corollary}\label{thm:weakmax}
    Let $u$ be the weak solution to \cref{eq:bvp2} constructed in \cref{thm:epslimit} with $f=0$, then
    \begin{equation*}
        \sup_{\partial^x_- \calD} u_{\Gamma} \leq \sup_{\calD} u \leq \sup_{\partial_\mathrm{hyp} \calD} g_+ .
    \end{equation*}
    In particular,
    \begin{equation*}
        \sup_{\partial^x_- \calD} \abs{u_{\Gamma}} \leq \sup_{\calD} \abs{u}\leq \sup_{\partial_\mathrm{hyp}\calD}\abs{g}.
    \end{equation*}
\end{corollary}
\begin{proof}
    This follows from \cref{eq:epsconverge} and \cref{thm:epsmax}.

\end{proof}

\section{The Perron-Wiener-Brelot solution}\label{sec:pwbsol}

Using the existence and regularity theory for product domains we can construct Perron-Wiener-Brelot solutions to the stationary Kramers-Fokker-Planck equation in divergence form with rough coefficients. Throughout this section we let \cref{assump:1,assump:posdiv} hold.

We first state the following convergence lemma.
\begin{lemma}[Convergence lemma]
    Let $\Omega\subset\R^{2n}$ be an open set and let $u_k$ be a locally uniformly bounded sequence of $\mathrm{K}$-harmonic functions (see \cref{def:kharmonic}) in $\Omega$.
    Then it has a subsequence that converges locally uniformly in $\Omega$ to a $\mathrm{K}$-harmonic function.
\end{lemma}
\begin{proof}
    Using the existence of $\mathrm{K}$-harmonic functions, see \cref{sec:perron}, and \cref{thm:holderint,thm:comparison} we can follow the proof of Lemma 3.4 in~\cite{KL96}.

\end{proof}
Next we record the following simple observation, i.e.~that the set of $\mathrm{K}$-superharmonic functions (see \cref{def:ksuperharmonic}) is closed under taking the minimum.
\begin{lemma}
    Let $\Omega\subset\R^{2n}$ be an open set.
    If $u$ and $w$ are $\mathrm{K}$-superharmonic in $\Omega$, then the function $\min(u, w)$ is $\mathrm{K}$-superharmonic.
\end{lemma}
We can extend the strong maximum principle \cref{thm:strong} to the following maximum principle for $\mathrm{K}$-superharmonic functions.
\begin{lemma}\label{thm:ksup:max}
    Let $\Omega\subset\R^{2n}$ be an open set.
    If $u$ is $\mathrm{K}$-superharmonic and if there is $\xi_0\in\Omega$ such that $u(\xi_0)=\inf_\Omega u$, then $u(\xi) = u(\xi_0)$ for all $\xi\in \overline{\mathcal{A}_{\xi_0}(\Omega)}$, where $\mathcal{A}_{\xi_0}(\Omega)$ is the attainable set from $\xi_0$ (see \cref{sec:strongmax}).
\end{lemma}
\begin{proof}
    For any $\xi'\in \mathcal{A}_{\xi_0}(\Omega)$ there is a curve $\eta:[0,T]\to\Omega$ such that $\eta(0) = \xi_0$ and $\eta(T)=\xi'$.
    We can cover the curve $\eta$ by cylinders $\calD_i\subset\Omega$, as in \cref{assump:D0}, for indices $i\in I$ and hence by compactness we can choose finitely many of them covering $\eta$, i.e.~$i=1, \dots, N$. Finally, we choose the last domain $\calD_N$ such that $\xi'\in \partial_{\mathrm{hyp}} \calD_N$, this is possible due to \cref{rmk:nonempty}.

    First, we have that $\xi_0\in \calD_1$.
    Let's consider the $\mathrm{K}$-harmonic function that satisfies (see \cref{sec:perron}) the following
    \begin{equation}\label{eq:ksup:max:1}
        \begin{cases}
            \L w = 0                             & \textrm{ in } \calD_1,                         \\
            \hphantom{\L} w = \theta_\varepsilon & \textrm{ on } \partial_{\mathrm{hyp}} \calD_1,
        \end{cases}
    \end{equation}
    where $\theta_\varepsilon:\overline{\calD_1}\to\R$, $\varepsilon>0$, is a sequence of continuous functions converging to $u$ pointwise, from below, as $\varepsilon\to 0$ and that $\inf_\Omega u\leq \theta_\varepsilon(\xi)\leq u(\xi)$ on $\partial_\mathrm{hyp} \calD_1$.
    Indeed, since $u$ is lower semi-continuous and $\overline{\calD_1}$ is compact, such $\theta_\varepsilon$ can be found.
    By the definition of $\mathrm{K}$-superharmonic functions, \cref{def:ksuperharmonic}, we have that
    \begin{equation*}
        w(\xi)\leq u(\xi) \textrm{ in } \calD_1.
    \end{equation*}
    On the other hand by the weak maximum principle, \cref{thm:weakunique}, we have
    \begin{equation*}
        w(\xi) \geq \inf_{\partial_\mathrm{hyp} \calD_1} \theta_\varepsilon \geq \inf_{\Omega} u \textrm{ in } \calD_1.
    \end{equation*}
    Hence, it follows that
    \begin{equation*}
        \inf_{\Omega} u = u(\xi_0) \geq w(\xi_0) \geq \inf_{\Omega} u,
    \end{equation*}
    which implies that
    \begin{equation*}
        \inf_{\calD_1} w = w(\xi_0) = u(\xi_0) = \inf_\Omega u.
    \end{equation*}
    Then by the strong maximum principle applied to $w$, \cref{thm:strong}, we obtain
    \begin{equation*}
        w(\xi) = u(\xi_0)=\inf_\Omega u \textrm{ for all } \xi\in \overline{\mathcal{A}_{\xi_0}(\calD_1)}.
    \end{equation*}
    Consequently, by \cref{eq:ksup:max:1}, we see that
    \begin{equation*}
        \theta_\varepsilon(\xi) = u(\xi_0) = \inf_\Omega u \textrm{ on } \partial_{\mathrm{hyp}} \calD_1 \cap \overline{\mathcal{A}_{\xi_0}(\calD_1)}.
    \end{equation*}
    By letting $\varepsilon\to 0$, we see that
    \begin{equation*}
        \theta(\xi) = u(\xi_0) = \inf_\Omega u \textrm{ on } \partial_{\mathrm{hyp}} \calD_1 \cap \overline{\mathcal{A}_{\xi_0}(\calD_1)}.
    \end{equation*}
    Specifically, by \cref{rmk:nonempty}, the above set is non-empty, and we know that $\eta$ intersects $\partial_{\mathrm{hyp}} \calD_1$ at some point $\xi_1$, i.e.~$u(\xi_1) = u(\xi_0)$. Repeating the argument above for all $i=1,\dots, N$ we obtain that $u(\xi') = u(\xi_0) = \inf_\Omega u$. This completes the proof.

\end{proof}
Recalling the definition of the upper and lower classes, \cref{def:uppersolution}, we get the following comparison principle between $\mathrm{K}$-superharmonic and $\mathrm{K}$-subharmonic functions as a consequence of \cref{thm:ksup:max}.
\begin{proposition}\label{thm:ksupsub:compare}
    Let $\Omega\subset\R^{2n}$ be a bounded open set.
    Suppose that $u$ is $\mathrm{K}$-superharmonic and $w$ is $\mathrm{K}$-subharmonic in $\Omega$.
    If $u$ and $w$ are bounded and if
    \begin{equation*}
        \limsup_{\xi\to \xi_0} w(\xi)\leq \liminf_{\xi\to \xi_0} u(\xi)
    \end{equation*}
    at each point $\xi_0\in\partial\Omega$,
    then $w\leq u$ in $\Omega$.
\end{proposition}
\begin{proof}
    Let's define $h(\xi) = u(\xi) -w(\xi)$.
    It follows that $h$ is $\mathrm{K}$-superharmonic and
    \begin{equation}\label{eq:ksupsub:compare:1}
        \liminf_{\xi\to \xi_0} h(\xi)\geq 0
    \end{equation}
    for each $\xi_0\in \partial \Omega$.

    Now suppose that there is a point $\xi'\in\Omega$ such that $h(\xi')<0$,
    then by the lower semi-continuity there is a point in $\Omega$, still denoted by $\xi'$, that $h(\xi') = \inf_\Omega h <0$.
    By \cref{thm:ksup:max} we get $h(\xi) = h(\xi')$ for all $\xi\in\overline{\mathcal{A}_{\xi'}(\Omega)}$.
    On the other hand, by \cref{rmk:nonempty}, we know there are points $\xi_0\in \partial\Omega\cap \overline{\mathcal{A}_{\xi'}(\Omega)}$ and that
    \begin{equation*}
        \liminf_{\xi \to \xi_0} h(\xi) \leq \liminf_{s\to T} h(\eta(s)) = h(\xi_0)=h(\xi') < 0,
    \end{equation*}
    where $\eta(s):[0,T]\to\Omega\cup\{ \xi_0\}$ is an admissible curve in the sense of \cref{eq:stationarycurve}, with $\eta(0) = \xi'$ and $\eta(T)=\xi_0$.
    This contradicts \cref{eq:ksupsub:compare:1} and hence the proof is complete.

\end{proof}
\begin{corollary}\label{thm:upperlower:compare}
    Let $\Omega\subset\R^{2n}$ be a bounded open set and let $g: \partial\Omega\to\R$ be a bounded function, then $\underline{H}_g \leq \overline{H}_g$.
\end{corollary}

\subsection{Proof of \cref{thm:pwbexist}}
Let $\Omega\subset \R^{2n}$ be an open set and let $\calD\ssubset\Omega$ be a cylinder.
If $u$ is $\mathrm{K}$-superharmonic in $\Omega$ and bounded on $\calD$, we define the \emph{$\mathrm{K}$-harmonic modification} (sometimes called harmonic lifting)
\begin{align*}
    \widetilde{u} =
    \begin{cases}
        u, & \textrm{ in } \Omega\setminus\overline{\calD}, \\
        w, & \textrm{ in } \calD,
    \end{cases}
\end{align*}
where
\begin{align*}
    w(\xi) = \sup\{ h(\xi): h\in \rmC(\overline{\calD}) \textrm{ is $\mathrm{K}$-harmonic and } h\leq u \textrm{ on } \partial_\mathrm{hyp}\calD \}.
\end{align*}
Then it is clear that $\widetilde{u}\leq u$ on $\Omega$.
Moreover, from \cref{thm:holderint}, $\widetilde{u}$ is $\mathrm{K}$-superharmonic in $\Omega$ and $\mathrm{K}$-harmonic in $\calD$.

Now we have all the ingredients to prove the existence of the Perron-Wiener-Brelot solution, \cref{thm:pwbexist}, following the proof of~\cite[Theorem 5.1]{KL96}.

\begin{flushright}
    \qedsymbol
\end{flushright}

\subsection{Resolutivity}
Next, we study the resolutivity of the Perron-Wiener-Brelot solution.

The following lemma is easy to verify from the definition of the lower and upper solution, \cref{def:uppersolution}.
\begin{lemma}
    Let $\Omega\subset\R^{2n}$ be a bounded open set.
    Let $f$ and $g$ be bounded functions on $\partial\Omega$ and let $c\geq 0$ be any constant. Then
    \begin{enumerate}
        \item $\overline{H}_c\leq c$.
        \item $\overline{H}_{cf}=c\overline{H}_f$ and $\underline{H}_{cf}=c\underline{H}_f$. If $f$ is resolutive, then $cf$ is resolutive and $H_{cf}=cH_f$.
        \item $\overline{H}_{-f}=-\underline{H}_f$. If $f$ is resolutive then so is $-f$ and $H_{-f} = -H_f$.
        \item If $f\leq g$, then $\overline{H}_f\leq \overline{H}_g$ and $\underline{H}_f\leq \underline{H}_g$.
        \item $\overline{H}_{f+g}\leq \overline{H}_f + \overline{H}_g$ and $\underline{H}_{f+g}\geq \underline{H}_f+\underline{H}_g$ whenever the sums are defined. If $f$ and $g$ are resolutive and $f+g$ is defined, then $f+g$ is resolutive and $H_{f+g} = H_f + H_g$.
    \end{enumerate}
\end{lemma}

Having all the tools at hand we can apply the proofs from~\cite{Ekl79} to obtain the following two lemmas, we include the proofs for the convenience of the reader.
\begin{lemma}\label{thm:pwblinearclass}
    Let $\Omega\subset\R^{2n}$ be a bounded open set.
    If $\{g_i\}$ is a sequence of bounded resolutive functions that converges uniformly to $g$, then $g$ is resolutive and $H_{g_i}$ converges uniformly to $H_g$.
\end{lemma}
\begin{proof}
    First let's denote by $\overline{H}_i=\overline{H}_{g_i}$ and similarly for $\underline{H}_i$ and $H_i$.
    Let $\varepsilon>0$, for sufficiently large $i$ we have $\abs{g-g_i}<\varepsilon$ on $\partial\Omega$ hence $g_i -\varepsilon<g<g_i +\varepsilon$.
    Therefore, $\abs{\overline{H}_g-\overline{H}_i}\leq \varepsilon$ on $\Omega$.
    Similarly, $\abs{\underline{H}_g -\underline{H}_i}\leq \varepsilon$ on $\Omega$.
    Since each $g_i$ is resolutive, it follows that $\underline{H}_i = \overline{H}_i$ and that
    \begin{equation*}
        \abs{\underline{H}_g - \overline{H}_g} = \abs{\underline{H}_g -\underline{H}_i + \overline{H}_i- \overline{H}_g}\leq \abs{\underline{H}_g -\underline{H}_i}+\abs{\overline{H}_g - \overline{H}_i}\leq 2\varepsilon.
    \end{equation*}
    By letting $\varepsilon\to 0$ we see $g$ is resolutive.

    The uniform convergence of $H_i$ to $H_g$ on $\Omega$ follows since $\abs{H_g - H_i} = \abs{\underline{H}_g - \underline{H}_i}\leq \varepsilon$ on $\Omega$ for sufficiently large $i$.

\end{proof}
\begin{lemma}\label{thm:pwbsubresolutive}
    Let $\Omega\subset\R^{2n}$ be an open set.
    If $u$ is a bounded $\mathrm{K}$-subharmonic function on the bounded domain $\Omega$ and $g(\xi_0) = \lim_{\xi\to \xi_0} u(\xi)$ exists for all $\xi_0\in\partial\Omega$, then $g$ is resolutive.
\end{lemma}
\begin{proof}
    Since $g$ is bounded, $\overline{H}_g$ and $\underline{H}_g$ are bounded and $\mathrm{K}$-harmonic on $\Omega$.
    Since $u\in \mathcal{L}_g$, we have by \cref{thm:upperlower:compare} that $u\leq \underline{H}_g$ hence
    \begin{equation*}
        \liminf_{\xi\to \xi_0} \underline{H}_g (\xi)\geq \liminf_{\xi\to \xi_0} u(\xi) = g(\xi_0)
    \end{equation*}
    for all $\xi_0\in\partial\Omega$.
    Therefore, $\underline{H}_g\in \mathcal{U}_g$ and hence $\underline{H}_g\geq \overline{H}_g$. From \cref{thm:upperlower:compare} we have that $g$ is resolutive.

\end{proof}

\begin{lemma}\label{thm:subisksub}
    Let $u\in\rmH^1_{\mathrm{hyp}}(\calD)$ be a lower semicontinuous weak subsolution, then $u$ is $\mathrm{K}$-subharmonic in $\calD$.
\end{lemma}
\begin{proof}
    Using the weak trace \cref{thm:weaktrace}, \cref{thm:comparison}, and the proof follows as in~\cite[Lemma 4.2]{KL96}.

\end{proof}

\begin{lemma} \label{thm:pwbapprox}
    Let $\Omega$ be a bounded open set in $\R^{2n}$ and let \cref{assump:1,assump:posdiv,assump:2} hold.
    If $g$ is any continuous function on $\overline{\Omega}$,
    then for any $\varepsilon>0$ there is a function $u$ which is the difference of two continuous $\mathrm{K}$-subharmonic functions on a domain containing $\overline{\Omega}$ such that $\sup_{\overline{\Omega}} \abs{u-g}<\varepsilon$.
\end{lemma}
\begin{proof}
    Here we need to assume that $\mathbf{A}$ is differentiable w.r.t.~$v$, i.e.~\cref{assump:2}, hence the operator \cref{eq:kfp} becomes
    \begin{equation}\label{eq:bvp:nice}
        \L u = \sum_{i,j=1}^n a_{ij}\partial_{v_i }\partial_{v_j} u + \widetilde{b}_j \partial_{v_j} u + v\cdot\nabla_x u,
    \end{equation}
    where $\widetilde{b}_j := b_j + \sum_{i=1}^n \partial_{v_i} a_{ij}$.

    As in~\cite{Ekl79}, we first approximate $g$ by a polynomial $u$ such that $\sup_{\overline{\Omega}} \abs{u-g} <\varepsilon$. It suffices to find a $\rmC^1(\overline{\calD})$ weak subsolution $w$ to $\L$ such that $w - u$ is again a weak subsolution to $\L$, since $u = w - (w-u)$.

    Indeed, we claim that
    \begin{equation*}
        \mathscr{L} e^{\delta v\cdot (x+\mathbf{q})} \geq C >0
    \end{equation*}
    for some $C$ and $\mathbf{q}\in\R^n$ by choosing $\delta$ large enough.
    Computing the left-hand side we get
    \begin{equation*}
        \mathscr{L} e^{\delta v\cdot (x+\mathbf{q})}
        = \delta^2 a_{ij} (x_i+q_i)(x_j+q_j) e^{\delta v\cdot (x+\mathbf{q})} + \widetilde{b}_j \delta (x_j +q_j) e^{\delta v\cdot (x+\mathbf{q})} + \delta \abs{v}^2 e^{\delta v\cdot (x+\mathbf{q})},
    \end{equation*}
    and it is easy to verify that as long as we choose $\mathbf{q}$ such that $\abs{x+\mathbf{q}}>C'>0$ for some constant $C'>0$ and also $\delta> 0$ large enough, then the above claim is true.

    Finally, choosing $\widehat c > 0$ large enough, we can take $w = \widehat c e^{\delta v\cdot (x+\mathbf{q})}$ which gives $\L(w-u) > 1$. This completes the proof.

\end{proof}

\subsubsection{Proof of \cref{thm:pwbresolutive}} \label{sec:theoremresolutive}
The following is an adaptation of the proof in~\cite{Ekl79}.

Since $\partial\Omega$ is compact, let $K$ be an open set containing $\overline{\Omega}$. Applying \cref{thm:pwbapprox} we obtain a function $u_\varepsilon$ satisfying: $u_\varepsilon = w- h$ where $w$ and $h$ are $\mathrm{K}$-subharmonic and continuous on $K$ and
$\sup_{\partial\Omega}\abs{u_\varepsilon- g}<\varepsilon$ for fixed $\varepsilon>0$.
Then by \cref{thm:pwbsubresolutive}, $w|_{\partial\Omega}$ and $h|_{\partial\Omega}$ are resolutive boundary functions and hence $u_\varepsilon|_{\partial\Omega} = w|_{\partial\Omega} - h|_{\partial\Omega}$ is also resolutive.
It follows from \cref{thm:pwblinearclass} that $g$ is resolutive as $g$ is a uniform limit of $u_\varepsilon|_{\partial\Omega}$.
\begin{flushright}
    \qedsymbol
\end{flushright}

\section{Barriers and boundary regularity for Perron's solution}\label{sec:barrier}
We shall define the barrier function for the boundary value problem as in classical potential theory.
It gives necessary and sufficient conditions for the regularity of boundary points.
We investigate the geometry of the boundary and give some examples of barrier functions.
Throughout this section, we let \cref{assump:1,assump:2,assump:posdiv} hold, hence the operator $\L$ has the non-divergence form of \cref{eq:bvp:nice}.

\begin{definition}\label{def:barrier}
    Suppose that $\xi_0$ is a boundary point of a bounded domain $\Omega\subset\R^n\times\R^n$.
    A function $w$ is a barrier in $\Omega$ at the point $\xi_0$ if
    \begin{enumerate}
        \item $w$ is positive and $\mathrm{K}$-superharmonic in $\Omega$,
        \item $\liminf_{\zeta\to\xi} w(\zeta) >0$ if $\xi\in\partial\Omega$, $\xi\neq\xi_0$,
        \item $\lim_{\zeta\to\xi_0} w(\zeta) =0$.
    \end{enumerate}
\end{definition}
As in the classical theory, this definition is completely local:
Let $\widetilde{\Omega}$ be another domain such that $\overline{B}\cap\widetilde{\Omega} = \overline{B}\cap \Omega$ for some open ball $B$ centered at $\xi_0$.
Suppose that there is a barrier, say $\widetilde{w}$, in $\widetilde{\Omega}$ at $\xi_0$.
Let $m = \inf\{\widetilde{w}(\xi):\xi\in\partial B\cap \widetilde{\Omega}\}$.
Then $m>0$ and it follows that the function
\begin{equation*}
    w=
    \begin{cases}
        \min(\widetilde{w}, m) & \textrm{ in } B\cap\widetilde{\Omega}, \\
        m,                     & \textrm{ in } \Omega\setminus B
    \end{cases}
\end{equation*}
is a barrier in $\Omega$.
Hence, there is a barrier in $\Omega$ at $\xi_0$ exactly when there is a barrier in $\widetilde{\Omega}$.
The proof of the following result is classical, see for instance~\cite[Theorem 6.1]{KL96}
\begin{proposition}
    Suppose that $g:\partial\Omega\to\R$ is bounded and continuous at $\xi_0\in\partial\Omega$.
    If there is a barrier in $\Omega$ at $\xi_0$, then
    \begin{equation*}
        \lim_{\xi\to\xi_0} \underline{H}_g (\xi ) = g(\xi_0) = \lim_{\xi\to\xi_0}\overline{H}_g (\xi).
    \end{equation*}
\end{proposition}

Finally, we give some examples of regular boundary parts for some domains.
\begin{example}\label{exp:barrier:1}
    A simple example is that if $\calD$ satisfies \cref{assump:D0}, then $\partial_\mathrm{hyp}\calD$ is regular by \cref{thm:2,thm:reg}.
\end{example}
\begin{example}\label{exp:barrier:2}
    Let $\Omega\subset\R^{2n}$ be a domain, we say $\Omega$ is \emph{velocity symmetric} if $(x,v)\in \Omega$ then $(x,-v)\in\Omega$.
    Now take a velocity symmetric $\rmC^2$ domain $\Omega\subset\R^{2n}$ and let $\mathbf{n}=(\mathbf{n}_x,\mathbf{n}_v)$ be its unit outer normal, then the boundaries of $\Omega$ can be identified as
    \begin{equation*}
        \partial^v \Omega := \{(x,v)\in\partial\Omega: \mathbf{n}_v\neq 0\},\quad
        \partial^x \Omega := \{(x,v)\in\partial\Omega: \mathbf{n}_v = 0\},
    \end{equation*}
    and similarly for $\partial^x_\pm \Omega$, $\partial^x_0 \Omega$ and $\partial_{\mathrm{hyp}}\Omega$ as in \cref{def:weak:bdy}.

    Consider any $\xi_0 = (x_0, v_0)\in\partial^x_0 \Omega$,
    we can set $\mathbf{n}_x = (0,\dots, 0, 1)$ by a change of coordinate and $\partial\Omega$ is locally a graph given by $F(x', v) = x_{n}$ where $x'= (x_1,\dots,x_{n-1})$ and $F$ is a function.
    We impose the condition that $F(x', v) \leq F(x', v_0)$ for all $v\in\R^n$ and for all $\xi_0\in\partial^x_0 \Omega$.

    We claim that all $\xi_0 = (x_0,v_0)\in\partial_\mathrm{hyp}\Omega$ are regular points.
    Indeed,
    \begin{itemize}
        \item For all $\xi_0 = (x_0, v_0)\in \partial^v\Omega \cup \partial^x_+ \Omega$ we can construct a barrier as in~\cites{Man97,DP06},
              \begin{equation*}
                  w(\xi) = e^{-\delta \abs{\mathbf{n}(\xi_0)}^2} - e^{-\delta \abs{\xi-\xi_0 -\mathbf{n}(\xi_0)}^2},
              \end{equation*}
              where $\delta=\delta(\xi_0)$ is a constant large enough.
        \item For any points $\xi_0\in\partial^x_0 \Omega$ it follows that we can put a product domain
              \begin{equation*}
                  \calD_{\xi_0} :=\begin{cases}
                      F(x', v_0) < x_n, \\
                      \abs{v} < R
                  \end{cases}
              \end{equation*}
              for some $R>0$ such that the cylinder is tangent to $\Omega$ at $\xi_0$.
              It is clear that $\calD_{\xi_0}$ satisfies \cref{assump:D0}, with $\partial \calU \in \rmC^{1,1}$.
              Let now $u$ be a $\mathrm{K}$-harmonic function (see \cref{sec:perron}) such that $u = w$ on $\partial_\mathrm{hyp} \calD_{\xi_0}$. From the strong maximum principle \cref{thm:strong}, we know that $u(\xi_0) = 0$ and $u(\xi)>0$ for all $\xi\in\calD_{\xi_0}$, hence $u$ is a barrier function for $\Omega$ at $\xi_0$.
    \end{itemize}
\end{example}
\begin{example}
    Let $B\subset\R^{2n}$ be a unit ball, then $B$ is everywhere regular.
    \begin{figure}[H]

        \includegraphics{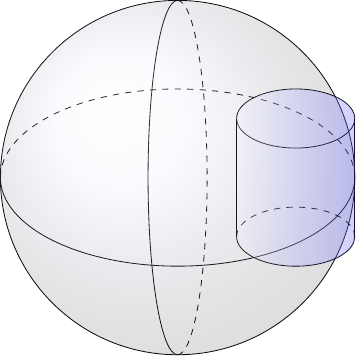}
    \end{figure}
    Indeed, observe that for any $\xi_0\in\partial B$, if $\mathbf{n}_v(\xi_0) = 0$ then $v_0 = 0$, and it is trivial to verify that $F(x',v)\leq F(x',0)$ is satisfied around all $(x_0, 0)\in\partial^x_0 B$.
    Hence, $B$ is everywhere regular.
\end{example}

\section*{Acknowledgments}
B.~Avelin and M.~Hou were supported by [Swedish Research Council dnr: 2019--04098].
The authors wish to thank Malte {Litsg\r ard} and Zhu Yuzhe for fruitful discussions.

\appendix

\section{Proof of \cref{thm:renormal}} \label{sec:appendix}

Recall \cref{def:weaksol}, we take the test function $\varphi(x, v) = \psi(y,v)\rho_k(y - 2\mathbf{n}(y)/k -x)$ for any fixed $y\in\overline{\calU}$, where $\psi(y,v)\in \rmC^\infty(\overline{\calD})$ with $\psi|_{\partial^v \calD} = 0$.
Note that we have $\varphi(x,v)=0$ when $x\in\partial\calU$ whenever $k$ is large enough, hence $\varphi$ is an admissible test function in \cref{def:weaksol}. Now, by integrating the weak formulation over $y$ we get
\begin{multline}\label{eq:contran:1}
    \iint_{\calD} \left(\bfA\nabla_v u\right)_{\star k} (y,v) \cdot\nabla_v \psi(y,v) \\
    - \left(\mathbf{b}\cdot\nabla_v u\right)_{\star k}(y,v)\psi(y,v) - (v\cdot\nabla_y u_{\star k}(y,v))\psi(y,v)\,\mathrm{d}y\mathrm{d}v \\
    = \iint_{\calD} f_{\star k}(y,v)\psi(y,v)\,\mathrm{d}y\mathrm{d}v
    - \iint_\calD R_{k}(y,v)\psi(y,v)\,\mathrm{d}y\mathrm{d}v
\end{multline}
where
\begin{equation*}
    R_k(y, v) = \frac{2}{k} \int_{\calU} u(x,v)  v\cdot(\nabla_y \mathbf{n}(y))^{\mathrm{T}} \nabla \rho_k (y-2\mathbf{n}(y)/k -x)\,\mathrm{d}x .
\end{equation*}
We can do integration by parts on the $v\cdot\nabla_y$ term to obtain
\begin{multline}\label{eq:contran:2}
    \iint_{\calD} \left(\bfA\nabla_v u\right)_{\star k} (y,v) \cdot\nabla_v \psi(y,v) - \left(\mathbf{b}\cdot\nabla_v u\right)_{\star k}(y,v)\psi(y,v) \\
    + u_{\star k}(y,v)v\cdot\nabla_y\psi(y,v)\,\mathrm{d}y\mathrm{d}v
    = \iint_{\calD} f_{\star k}(y,v)\psi(y,v)\,\mathrm{d}y\mathrm{d}v
    \\
    - \iint_\calD R_{k}(y,v)\psi(y,v)\,\mathrm{d}y\mathrm{d}v + \int_{\partial^x \calD} v\cdot\mathbf{n}_x u_{\star k}(y,v) \psi(y,v)\,\mathrm{d}S(y,v).
\end{multline}
It follows from the same argument from the proof of \cref{thm:contran:converge} that for all $\psi\in \rmC^\infty(\overline{\calD})$ with $\psi|_{\partial^v \calD}=0$ we have
$\iint_\calD R_k (y,v) \psi(y,v)\,\mathrm{d}y\mathrm{d}v \to 0 \textrm{ as } k\to\infty$.

Next, we rewrite \cref{eq:contran:1} as follows
\begin{multline}\label{eq:contran:3}
    \iint_{\calD} \bfA\nabla_v u_{\star k}  \cdot\nabla_v \psi - \mathbf{b}\cdot\nabla_v u_{\star k}\psi - (v\cdot\nabla_x u_{\star k})\psi\,\mathrm{d}x\mathrm{d}v \\
    = \iint_{\calD} f_{\star k}\psi\,\mathrm{d}x\mathrm{d}v
    - \iint_\calD R_{k}\psi\,\mathrm{d}x\mathrm{d}v - \iint_\calD E_k\psi\,\mathrm{d}x\mathrm{d}v
\end{multline}
where
\begin{equation*}
    E_k = \left(\bfA\nabla_v u\right)_{\star k} - \bfA\nabla_v u_{\star k} - \left(\mathbf{b}\cdot\nabla_v u\right)_{\star k} + \mathbf{b}\cdot\nabla_v u_{\star k}.
\end{equation*}
By the properties of the convolution-translation, it holds that $E_k \to 0$ in $\rmL^2 (\calD)$ as $k\to\infty$.

For the boundary term, we know from \cref{thm:contran:converge} together with the weak trace operator \cref{eq:weaktrace} that $u_{\star k}|_{\partial^x\calD} \to \tr_x(u)$ strongly in $\rmL^2_{\mathrm{loc}}(\partial^x\calD, \abs{v\cdot\mathbf{n}_x}^2)$.
Thus, using Cauchy-Schwartz it holds that for any $\varphi\in \rmC^1(\overline{\calD})$ and $\eta(v)\in \rmC_c^\infty(\calV)$ a non-negative cutoff function that
    {
        \begin{multline} \label{eq:contran:4}
            \abs*{\int_{\partial^x \calD} v \cdot n_x (u_{\star k}-\tr_x(u)) \varphi\eta \,\rmd S }
            \\
            \leq
            \left ( \int_{\partial^x\calD} \abs{v \cdot n_x}^2 \abs{u_{\star k} - \tr_x(u)}^2 \eta\, \rmd S \right )^{\frac{1}{2}} \left ( \int_{\partial^x\calD} \abs{\varphi}^2 \eta\, \rmd S \right )^{\frac{1   }{2}} \to 0.
        \end{multline}}
Combing the above with \cref{eq:contran:2}, we see that for $\psi \in \rmC^1(\overline{\calD})$ with $\psi|_{\partial^v \calD} = 0$,
\begin{equation*}
    \iint_{\calD} \bfA\nabla_v u \cdot\nabla_v \psi - \mathbf{b}\cdot\nabla_v u\psi - u(v\cdot \nabla_x \psi) \,\mathrm{d}x\mathrm{d}v = \iint_{\calD} f\psi\,\mathrm{d}x\mathrm{d}v + \int_{\partial^x \calD} v\cdot\mathbf{n}_x \tr_x(u) \psi\,\mathrm{d}S.
\end{equation*}

To prove the renormalized Green's formula \cref{eq:renormal} we choose $\psi(x, v) = \Phi'(u_{\star k})\phi(x,v)$ in \cref{eq:contran:3} where $\Phi$ and $\phi$ are test functions as stated in the Lemma.
Note here, if we choose $\Phi$ to satisfy condition~\eqref{item:renormal1}, then for each fixed $k$, $\Phi'(u_{\star k})\in \rmH^1(\calD)$ such that $\mathrm{Tr}_v(\Phi'(u_{\star k})) = 0$, and we can approximate it by $\rmC^1(\overline{\calD})$ functions that vanishes on $\partial^v \calD$, say $\varphi_{k,l}$, in $\rmH^1(\calD)$ norm. Hence, we obtain an equation with $\varphi_{k,l}\phi$ as test functions. By sending $l\to\infty$ for fixed $k$ and by observing $\nabla_v \varphi_{k,l}\to \Phi''(u_{\star k})\nabla_v u_{\star k}$ in $\rmL^2(\calD)$, it follows that
\begin{multline*}
    \iint_{\calD} \Phi'(u_{\star k}) \bfA\nabla_v u_{\star k}  \cdot\nabla_v \phi + \Phi''(u_{\star k})\phi \left(\bfA \nabla_v u_{\star k}\right)\nabla_v u_{\star k} - \mathbf{b}\cdot\nabla_v \left(\Phi(u_{\star k})\right) \phi\,\mathrm{d}x\mathrm{d}v\\
    - \iint_\calD (v\cdot\nabla_x \left(\Phi(u_{\star k})\right))\phi\,\mathrm{d}x\mathrm{d}v
    = \iint_{\calD} f_{\star k}\Phi'(u_{\star k})\phi\,\mathrm{d}x\mathrm{d}v
    - \iint_\calD R_{k}\Phi'(u_{\star k})\phi\,\mathrm{d}x\mathrm{d}v \\
    - \iint_\calD E_k\Phi'(u_{\star k})\phi\,\mathrm{d}x\mathrm{d}v.
\end{multline*}
By an integration by parts on the $v\cdot\nabla_x$ term, we derive
\begin{multline*}
    \iint_{\calD} \Phi'(u_{\star k}) \bfA\nabla_v u_{\star k}  \cdot\nabla_v \phi + \Phi''(u_{\star k})\phi \left(\bfA \nabla_v u_{\star k}\right)\nabla_v u_{\star k} - \mathbf{b}\cdot\nabla_v \left(\Phi(u_{\star k})\right) \phi\,\mathrm{d}x\mathrm{d}v\\
    + \iint_\calD \Phi(u_{\star k})v\cdot\nabla_x \phi\,\mathrm{d}x\mathrm{d}v - \int_{\partial^x \calD} v\cdot\mathbf{n}_x \Phi(u_{\star k})\phi\,\mathrm{d}S \\
    = \iint_{\calD} f_{\star k}\Phi'(u_{\star k})\phi\,\mathrm{d}x\mathrm{d}v
    - \iint_\calD R_{k}\Phi'(u_{\star k})\phi\,\mathrm{d}x\mathrm{d}v
    - \iint_\calD E_k\Phi'(u_{\star k})\phi\,\mathrm{d}x\mathrm{d}v.
\end{multline*}
By the convergence of $R_k$ and $E_k$ and the convergence of the boundary term in \cref{eq:contran:4}, we can let $k\to\infty$ to complete the proof of \cref{eq:renormal}.

We also note that if $u$ is bounded then we only need to know that $\Phi \in \rmW^{2,\infty}_{\textrm{loc}}$ to prove \cref{eq:renormal}.

If $\partial \calU$ is only $\rmC^{0,1}$, then we need to carefully approximate $\calU$ by $\rmC^{1,1}$ domains so that the convolution-translation can be defined, and also a trace function $u_\Gamma\in \rmL^2_{\mathrm{loc}}(\partial^x\calD, \abs{v\cdot\mathbf{n}_x}^2 )$ can be constructed, for details we refer to~\cite[Lemma 2.3]{Zhu22}.
\begin{flushright}
    \qedsymbol
\end{flushright}

\section{Proof of \cref{thm:weaktrace}}\label{sec:appendix:weaktrace}

For the sake of completeness, we present the proof of \cref{thm:weaktrace}, of which the first part can be found in~\cite{Sil22}.

\newcommand{\noncontentsline}[3]{}
\newcommand{\tocless}[2]{\bgroup\let\addcontentsline=\noncontentsline#1{#2}\egroup}
\tocless{\subsection*{Proof of \cref{thm:weaktrace}}}
First, let $u\in \rmC^1(\calD)\cap \rmC(\overline{\calD})$, consider $v\cdot\mathbf{n}(x) u \eta(v)$ where $\eta(v) \in \rmC_0^1(\calV)$ is non-negative.
By integration by parts, we have
\begin{equation*}
    \iint_{\calD} (v\cdot\nabla_x u) ( v\cdot\mathbf{n}(x) u \eta) \,\rmd x\rmd v = \int_{\partial^x \calD} \abs{v\cdot\mathbf{n}_x}^2 u^2 \eta \,\rmd S - \iint_{\calD} u (v\cdot\nabla_x (v\cdot\mathbf{n}(x) u \eta))\,\rmd x\rmd v.
\end{equation*}
It follows that
\begin{equation*}
    \int_{\partial^x \calD} \abs{v\cdot\mathbf{n}_x}^2 u^2 \eta \,\rmd S = 2\iint_{\calD} (v\cdot\nabla_x u) ( v\cdot\mathbf{n}(x) u \eta) \,\rmd x\rmd v
    + \iint_{\calD} u^2 \eta v\cdot\nabla_x (v\cdot\mathbf{n}(x) )\,\rmd x\rmd v,
\end{equation*}
which implies by duality that
\begin{multline*}
    \int_{\partial^x \calD} \abs{v\cdot\mathbf{n}_x}^2 u^2\eta \,\rmd S \leq \norm{v\cdot\mathbf{n}(x)u\eta}_{\rmL^2(\calU;\rmH^1(\calV))}\norm{v\cdot\nabla_x u}_{\rmL^2(\calU;\rmH^{-1}(\calV))} \\
    + \sup_{\calD}\abs{\eta v\cdot\nabla_x(v\cdot\mathbf{n}(x))} \norm{u}_{\rmL^2(\calD)}^2
    \leq C \norm{u}_{\rmH^1_{\mathrm{hyp}}}^2,
\end{multline*}
where $C$ is a constant independent of $u$.
Thus, by density of smooth functions, \cref{thm:contran:converge}, we prove \cref{eq:weaktrace}.

We now prove the integration by parts formula, \cref{eq:intbypart}, under the assumption that $\partial \calV$ is $\rmC^{1,1}$ for simplicity.
Let $w \in \rmC^1(\overline{\calD})$ and let $\varphi \in \rmC^1(\overline{\calD})$ such that $\varphi|_{\partial^v \calD} = 0$.
By an integration by parts we have
\begin{equation*}
    \iint_{\calD} (v\cdot\nabla_x w) \varphi \,\rmd x\rmd v = \int_{\partial^x \calD} v \cdot \mathbf{n}(x) w \varphi \,\rmd S - \iint_{\calD} w (v\cdot\nabla_x \varphi)\,\rmd x\rmd v.
\end{equation*}
First we claim that there exists a function $\eta(v) \in \rmC_0^1(\calV)$, $0 \leq \eta \leq 1$ such that $\frac{\varphi^2(x,v)}{\eta(v)} \in L^\infty(\partial^x \calD)$. Indeed, we can take $\eta(v) = \min\{1, \text{dist}(v, \partial \calV)\}$, because $\varphi$ is globally Lipschitz continuous we immediately have that $\frac{\varphi^2(x,v)}{\eta(v)} \in L^\infty(\partial^x \calD)$.
\begin{align*}
    \int_{\partial^x \calD} |v \cdot \mathbf{n}(x) w \varphi| \,\rmd S
     & =
    \int_{\partial^x \calD} \left |v \cdot \mathbf{n}(x) w \sqrt{\eta(v)} \frac{\varphi(x,v)}{\sqrt{\eta(v)}} \right | \,\rmd S
    \\
     & \leq
    \left (\int_{\partial^x \calD} |v \cdot \mathbf{n}(x)|^2 w^2 \eta(v) \rmd S\right )^{1/2}
    \left (\int_{\partial^x \calD} \frac{\varphi^2(x,v)}{\eta(v)} \rmd S \right )^{1/2}.
\end{align*}
Now take $w = u_{\star k}-u_{\star l}$, then we have
\begin{equation*}
    \int_{\partial^x \calD} |v \cdot \mathbf{n}(x)| |u_{\star k} - u_{\star l}| |\varphi| \,\rmd S
    \leq C \norm{u_{\star k}-u_{\star l}}_{\rmH^1_{\mathrm{hyp}}}^2.
\end{equation*}
Now by \cref{thm:contran:converge} we have $u_{\star k}-u_{\star l} \to 0$ in $\rmH^1_{\textrm{hyp}}$, as $k,l\to\infty$. Which implies that up to a subsequence $u_{\star k} \to \tr_x(u)$ a.e.~on $\partial^x \calD$, which proves \cref{eq:intbypart}.
\begin{flushright}
    \qedsymbol
\end{flushright}

\section{Notes on resolutivity in the time-dependent case}
In this section we consider an alternate proof of \cref{thm:pwbapprox} in the time-dependent case (see \cref{eq:time}), where we can dispense with \cref{assump:2}. We first note that all the necessary results in \cref{sec:perron,sec:pwbsol} can be extended to the time-dependent case with minor modifications. The corresponding sub/super solutions defined using the comparison principle with respect to continuous solutions of $\partial_t u - \L u = 0$ are called $\mathrm{K}$-sub/super-solutions in this section. We also note that the weak formulation of the time-dependent equation can be extracted from \cref{def:rensub}.

\begin{lemma}\label{thm:pwbapproxtime}
    Let $\Omega\subset\R^{2n+1}$ be a bounded open set.
    If $g$ is any continuous function on $\overline{\Omega}$.
    Then, for any $\varepsilon>0$ there exists a function $u$ which is the difference of two smooth $\mathrm{K}$-subsolutions on a domain containing $\overline{\Omega}$ such that $\sup_{\overline{\Omega}} \abs{u-g}<\varepsilon$.
\end{lemma}
\begin{proof}
    As in \cref{thm:pwbapprox}, we first approximate $g$ by a polynomial $u$ such that $\sup_{\overline{\Omega}} \abs{u-g} <\varepsilon$.
    Then observe that $(\partial_t-\mathscr{L})(t+1) = 1> 0$ weakly and set
    \begin{equation*}
        u = \left( u - C (t+1)\right) -(-C (t+1)),
    \end{equation*}
    where $C>0$ is a constant to be determined.

    Now to show that $u-C(t+1)$ is a $\mathrm{K}$-subsolution we first verify that it is a local weak solution. To begin we consider a cylinder $\calD_{(t_1,t_2)} := (t_1,t_2)\times\calD \supset \overline{\Omega}$ and a smooth function $\varphi\in \rmC_c^\infty(\calD_{(t_1,t_2)})$ to see that
    \begin{align*}
        [\partial_t u - \L u](\varphi)
        :=       &
        \iiint_{\calD_{(t_1,t_2)}} \left(\partial_t u \varphi + (\bfA \nabla_v u) \cdot \nabla_v \varphi -(\mathbf{b}\cdot\nabla_v u) \varphi - v \cdot \nabla_x u \varphi \right)\, \mathrm{d}z
        \\
        \lesssim &
        \norm{\varphi}_{\rmL^1((t_1,t_2)\times\calU; \mathrm{W}^{1,1}(\calV))} \lesssim \norm{\varphi}_{\rmL^2((t_1,t_2)\times\calU; \rmH^1(\calV))},
    \end{align*}
    as such $\partial_t u - \L u \in \rmL^2((t_1,t_2)\times\calU; \rmH^{-1}(\calV))$.
    If we consider $w = u - C(t+1)$ and $\varphi \geq 0$, we have that for $C$ large enough by the $\rmL^1$-Poincaré inequality that
    \begin{align*}
        [\partial_t w - \L w](\varphi) \leq 0
    \end{align*}
    as a distribution, hence it is also a weak subsolution in the sense of \cref{def:weaksol}. We represent its action on $\varphi$ as
    \begin{align*}
        [\partial_t w - \L w](\varphi) =: \iiint_{\calD_{(t_1,t_2)}} f \varphi \, \mathrm{d}z
    \end{align*}
    where $f\in \rmL^2((t_1,t_2)\times\calU; \rmH^{-1}(\calV))$.
    As such, since $w$ is smooth and bounded, from the time-dependent version of \cref{thm:subisksub} we see that $w$ is a $\mathrm{K}$-subsolution.

\end{proof}

\printbibliography

\end{document}